\documentclass{amsart}
\usepackage{setspace}
\usepackage{times}
\usepackage{graphicx}
\usepackage{wrapfig}
\usepackage{amsmath, amsthm, amssymb, amsfonts}
\usepackage{enumitem}
\usepackage{caption}
\usepackage{subcaption}
\usepackage{bbm}
\usepackage{mathtools}
\usepackage{xypic}

\everymath{\displaystyle}
\newcommand{\abs}[1]{\left| #1 \right|}
\newcommand{\fatn}[1]{f((T')^nx,\sigma_{n,x}(#1))}

\newtheorem{thm}{Theorem}
\newtheorem*{thm*}{Theorem}
\newtheorem{cor}{Corollary}
\newtheorem{lem}{Lemma}
\newtheorem*{lem*}{Lemma}
\newtheorem{prop}{Proposition}

\theoremstyle{remark}
\newtheorem{remark}{Remark}
\newtheorem{question}{Question}

\theoremstyle{definition}
\newtheorem*{defin*}{Definition}
\newtheorem{defin}{Definition}

\setenumerate{topsep=-\baselineskip}

\newcommand{\norm}[1]{\left\lVert#1\right\rVert}
\renewcommand{\epsilon}{\varepsilon}
\newcommand{\bbe}{\mathbb{E}}

\title{Generic properties of extensions}
\author{Mike Schnurr}
\address{Max Planck Institute for Mathematics in the Sciences, Inselstr. 22, 04103 Leipzig, Germany}
\email{schnurr@mis.mpg.de}

\begin{document}
	
\begin{abstract}
	Motivated by the classical results by Halmos and Rokhlin on the genericity of weakly but not strongly mixing transformations and the Furstenberg tower construction, we show that weakly but not strongly mixing extensions on a fixed product space with both measures non-atomic are generic. In particular, a generic extension does not have an intermediate nilfactor.
\end{abstract}

\maketitle
	
\section{Introduction} \label{Sec:Introduction}
The classical results by Halmos \cite{HalmosPaper} and Rokhlin \cite{Rokhlin} state that a ``typical'' measure-preserving transformation on a probability space $(X,\mu)$ is weakly but not strongly mixing. More precisely, the set of weakly mixing transformations is a dense, $G_{\delta}$ (hence residual) set for the weak topology and the set of strongly mixing transformations are of first category. The combination of these two results proved the existence of a weakly but not strongly mixing transformation, without providing a concrete example (for such an example, see \cite{Chacon}). Since then, much research has been done in finding ``typical'' properties of measure-preserving dynamical systems. See, e.g., \cite{Nadkarni}, \cite{KatokStepin}, \cite{Ageev1}, \cite{King}, \cite{Ageev2}, \cite{AlpernPrasad}, \cite{Ageev3}, \cite{RueLazaro}, \cite{Ageev4}, \cite{Solecki}, \cite{Guiheneuf}.

More than thirty years later, in completely unrelated efforts, Furstenberg \cite{FurstenbergOriginal} presented his celebrated ergodic theoretic proof of Szemeredi's Theorem on the existence of arbitrarily long arithmetic progressions in large subsets of $\mathbb{N}.$ While the original proof by Furstenberg used diagonal measures, the alternative proof by him, Katznelson, and Ornstein \cite{Furstenberg} building up a tower of so-called compact and weakly mixing extensions had much greater impact on further development of ergodic theory. This method of finding the characteristic factor has been extended to various ergodic theorems and is an active area of research, see, e.g., \cite{Chu}, \cite{ChuFrantzinakisHost}, \cite{AssaniPresser}, \cite{EisnerZorin}, \cite{BergelsonTaoZiegler}, \cite{FrantzinakisZorin}, \cite{EisnerKrause}, \cite{TaoZiegler} \cite{AssaniDuncanMoore}, \cite{Robertson}. For example, the correct characteristic factor for norm convergence of multiple ergodic averages was identified by Host and Kra \cite{HostKra} and has the structure of an inverse limit of nilsystems, see also Ziegler \cite{Ziegler}.

The purpose of this paper is to prove  analogues of the Halmos and Rokhlin category theorems for extensions (see Theorems \ref{Thm:CategoryTheorem} and \ref{Thm:FirstCategoryTheorem}), extending a result of Robertson on compact group extensions \cite{EARobertson}. Inspired by Rokhlin's skew product representation theorem (see, for example, \cite[p.69]{Glasner}), we consider extensions defined on product spaces with the natural projection as the factor map. We show that for a fixed product space, where both measures are non-atomic, a ``typical'' extension is weakly but not strongly mixing. Here, by an extension we mean an invertible extension of some (non-fixed), invertible, measure-preserving transformation on the factor. Note that the set of extensions is a closed, nowhere dense set in all invertible transformations on the product space (see Proposition \ref{Prop:GXIsClosed}), so the classical Halmos and Rokhlin results cannot be applied. The proof for weakly mixing extensions is a non-trivial adaptation of the original construction by Halmos. In particular, a ``typical'' extension does not have an intermediate nilfactor. For examples of systems lacking non-trivial nilfactors, see \cite{HostKraMaass}.

The paper is organized as follows. After discussing some preliminaries in Section \ref{Sec:Preliminaries}, we consider the case of discrete extensions in Section \ref{Sec:Discrete}, in particular showing that there are no weak mixing extensions on these spaces (Proposition \ref{Prop:NoWeaklyMixingFinite}), but that permutations are dense (Theorem \ref{Thm:DensityOfPermutationsFinite}). We then prove the Weak Approximation Theorem for Extensions on the unit square (Theorem \ref{Thm:WATE}) in Section \ref{Sec:WAT} using the density result for discrete extensions mentioned above. In Section \ref{Sec:UniformApproximation} we generalize a few results, including Halmos' Uniform Approximation Theorem, which are necessary to prove our Conjugacy Lemma for Extensions (Lemma \ref{Lem:ConjugacyLemmaExtensions}) in Section \ref{Sec:ConjugacyLemma}. Section \ref{Sec:CategoryTheorem} is devoted to the proof that weakly mixing extensions on the unit square are residual and Section \ref{Sec:General} addresses the case of general vertical measure. In Section \ref{Sec:StrongMixing} we define strongly mixing extensions, and show that such extensions are of first category (Theorem \ref{Thm:FirstCategoryTheorem}). Finally in Section \ref{Sec:Questions} we formulate some open questions. After this paper was made public, the first part of Question \ref{Que:FixedFactor} was answered by Eli Glasner and Benjamin Weiss, see \cite{GlasnerWeiss}.

\textbf{Acknowledgments.} The question of ``typical'' behavior of extensions was asked by Terence Tao for a fixed factor, motivated by \cite{HostKraMaass}, cf. Question \ref{Que:FixedFactor} and note in particular the following discussion on the difficulties with fixed factors. The author is very grateful to him for the inspiration. The author also thanks Tanja Eisner for introducing him to the problem and for many helpful discussions. The author is further thankful to Ben Stanley for being available to exchange ideas, the referee of this paper for careful reading and valuable comments which improved the paper, and to Bryna Kra, Michael Lin, Philipp Kunde, and Yonatan Gutman for helpful remarks. Lastly, the support of the Max Planck Institute is greatly acknowledged.

\section{Preliminaries} \label{Sec:Preliminaries}
As explained in the introduction, in this paper we will be working with extensions on product spaces through the natural projection. To be more precise, we let $(X, m)$ be a non-atomic standard probability space, $(Y, \eta)$ be a probability space, $(Z,\mu) = (X \times Y, m \times \eta),$ and $T,T'$ be measure-preserving transformations on $(Z,\mu), (X,m)$ respectively, such that $(Z,\mu,T)$ is an extension of $(X,m,T')$ through the natural projection map $\pi:Z \to X$ onto the first coordinate. We will assume throughout that $T,T'$ are invertible, and will identify two transformations if they differ on a set of measure zero. We will say ``$T$ is an extension of $T'$'' or ``$T$ extends $T'$'' if $T$ and $T'$ satisfy all conditions stated above. Throughout this paper, we will assume without loss of generality that $X$ is the unit interval and $m$ is the Lebesgue measure. We can assume this because all non-atomic standard probability spaces are isomorphic (see \cite[p. 61]{HalmosPaper}).

Let $\mathcal{G}(Z)$ denote the set of all invertible, measure-preserving transformations on $(Z,\mu)$ and let $\mathcal{G}_X = \{T \in \mathcal{G}(Z): \exists \ T' \in \mathcal{G}(X)  \text{ s.t. } T \text{ extends } T' \}$. Note that if we say $T \in \mathcal{G}_X,$ we assume that the transformation on the factor will be notated by $T'$. Further note that we will also write $\mathcal{G}_X$ to denote the corresponding set of Koopman operators.

The weak topology on $\mathcal{G}(Z)$ is the topology defined by the subbasic neighborhoods

$$N_{\epsilon}(T;E) = \{S \in \mathcal{G}(Z): \mu(TE \triangle SE) < \epsilon \},$$

where $\epsilon >0$ and $E$ is some measurable subset of $Z$. Note that if $Z$ is, say, the unit square with the Lebesgue measure, then it is sufficient for a subbasis to consider only dyadic sets (i.e., a finite union of dyadic squares). See \cite{HalmosLectures} for discussions of this topology. It is helpful to note that the weak topology happens to coincide with the weak (and strong) operator topology for the corresponding Koopman operators. Further, in this paper we will be interested in the weak topology on $\mathcal{G}_X,$ by which we mean the subspace topology inherited by the weak topology.

We will need the following two metrics on $\mathcal{G}(Z)$ defined by

\begin{align*}
d(S,T) &:= \sup_E \mu(SE \triangle TE) \\
d'(S,T) &:= \mu \left(\{z \in Z : Sz \neq Tz \}\right)
\end{align*}

where the $\sup$ in the first definition is taken over all measurable sets $E$. These metrics were used by Halmos in his proof of the category theorem, see \cite{HalmosLectures}. We note that both metrics induce the same topology on $\mathcal{G}(Z)$, but that topology is not the weak topology. Moreover, they satisfy $d(S,T) \le d'(S,T)$ for all $S,T \in \mathcal{G}(Z).$ The last important note is that $d'$ is invariant under multiplication by transformations. That is to say, for all $R,S,T \in \mathcal{G}(Z),$

$$d'(RS,RT) = d'(S,T) = d'(SR,TR).$$

Let $L^2(Z|X)$ denote the Hilbert module over $L^{\infty}(X).$ More precisely, for $f \in L^2(Z),$

$$ f \in L^2(Z|X) \text{ if and only if } \bbe(\abs{f}^2|X)^{1/2} \in L^{\infty}(X),$$

where $\bbe(f | X)$ is the conditional expectation of $f$ with respect to $X$. More specifically, it is the conditional expectation with respect to $\mathcal{A}:= \{\pi^{-1}(A) : A \in \mathcal{L} \},$ where $\mathcal{L}$ is the Lebesgue sigma algebra on $X$.  Let
	
$$\norm{f}_{L^2(Z|X)} :=  \bbe(\abs{f}^2|X)^{1/2}$$

and

$$\langle f,g \rangle_{L^2(Z|X)} := \bbe(f \overline{g} |X). $$

For more on $L^2(Z|X),$ see \cite{Tao}. One important result of $L^2(Z|X)$ that we do wish to emphasize for later is the Cauchy-Schwarz Inequality.

\begin{prop} \label{Thm:C-S}
Let $f,g \in L^2(Z|X).$ Then

$$\abs{\langle f,g \rangle_{L^2(Z|X)}} \le \norm{f}_{L^2(Z|X)} \norm{g}_{L^2(Z|X)} \text{ a.e.}$$
\end{prop}

Next we give a definition for weakly mixing extensions, cf. \cite{Tao}.

\begin{defin} \label{Def:WeakMixingExtension}  
	An extension $T$ of $T'$ is said to be \textit{weakly mixing} if for all $f,g \in L^2(Z|X),$
	
	$$ \lim_{N \to \infty} \frac{1}{N} \sum_{n=0}^{N-1} \norm{\bbe(T^nf \overline{g} |X) - (T')^n \bbe(f|X) \bbe(g|X)}_{L^2(X)} = 0.$$	
\end{defin}

For other possible (equivalent) definitions of weakly mixing extensions, see \cite[p.192]{Glasner}.

We denote by $\mathcal{W}_X \subset \mathcal{G}_X$ the set of weakly mixing extensions on $Z$.

We finally prove that the Baire Category Theorem is applicable to $\mathcal{G}_X,$ and further that $\mathcal{G}_X$ is topologically a small subset of $\mathcal{G}(Z).$

\begin{prop} \label{Prop:GXIsClosed}
Suppose $Y$ has more than one point. Then $\mathcal{G}_X$ is closed and nowhere dense in $\mathcal{G}(Z).$
\end{prop}

\begin{proof}
We first prove that $\mathcal{G}_X$ is closed. Let $T \in \mathcal{G}(Z) \backslash \mathcal{G}_X.$ We wish to find a neighborhod of $T$ that is disjoint from $\mathcal{G}_X.$ To this end, let $E \subset Z$ be a cylinder set (to be precise, $E$ is of the form $D \times Y$ for some measurable $D \subset X$), such that $TE$ is not a cylinder set, even up to measure zero. Define $M := \mu(E).$ Then for all cylinder sets $C$ with $\mu(C)=M, \mu(TE \triangle C) > 0.$

We claim that indeed $\inf_C \mu(TE \triangle C) > 0,$ where the $\inf$ is taken over all cylinder sets with measure exactly $M$. Suppose to the contrary that $\inf_C \mu(TE \triangle C)=0$. Let $C_n$ be a sequence of such cylinder sets such that not only $\mu(TE \triangle C_n) \to 0,$ but further such that

$$\sum_{n=1}^{\infty} \mu(TE \backslash C_n) < \infty. $$

Define

$$\hat{C} := \bigcup_{n=1}^{\infty} \bigcap_{m \ge n} C_m.$$

We claim that $\mu(TE \triangle \hat{C}) = 0.$ As $\hat{C}$ is clearly a cylinder set, we will arrive at a contradiction.

First consider 

$$\hat{C} \backslash TE = \left(\bigcup_{n=1}^{\infty} \bigcap_{m \ge n} C_m \right) \backslash TE = \bigcup_{n=1}^{\infty} \bigcap_{m \ge n} (C_m \backslash TE).$$

Now because $\mu(C_m \triangle TE) \to 0, \mu\left(\bigcap_{m \ge n} (C_m \backslash TE) \right) = 0$ for all $n.$ But then

$$\mu(\hat{C} \backslash TE) \le \sum_{n=1}^{\infty} \mu\left(\bigcap_{m \ge n} (C_m \backslash TE) \right) = 0.$$

On the other hand, 

$$TE \backslash \hat{C} = TE \backslash \left(\bigcup_{n=1}^{\infty} \bigcap_{m \ge n} C_m \right) = \bigcap_{n=1}^{\infty} \bigcup_{m \ge n} (TE \backslash C_m). $$

But by assumption, $\sum_{n=1}^{\infty} \mu(TE \backslash C_n) < \infty,$ so by the Borel-Cantelli lemma, $\mu(TE \backslash \hat{C})=0.$

Now, let $\epsilon := \inf_C \mu(TE \triangle C).$ We claim that for any $S \in \mathcal{G}_X, S \notin N_{\epsilon}(T;E).$ Indeed, $SE$ is (up to a null set) a cylinder set, and $\mu(SE)=M,$ so $\mu(TE \triangle SE) \ge \epsilon$ by definition of $\epsilon$.

Now because $\mathcal{G}_X$ is closed, in order to prove that it is nowhere dense, it is sufficient to show that $\mathcal{G}(Z) \backslash \mathcal{G}_X$ is dense. Fix $T \in \mathcal{G}_X$, let $\epsilon > 0$ and let

$$N_{\epsilon}(T) = \{S \in \mathcal{G}(Z): \mu(TE_i \triangle SE_i) < \epsilon, i = 1, \ldots, n \},$$

where $E_i$ are measurable sets. Now let $A \subset Z$ be a measurable set such that $0 \mu(A) < \epsilon,$ and $A$ is not a cylinder set. Further let $B \subset Z$ be a cylinder set such that $\mu(A \cap B) = \mu(A \cap B^c),$ and define $A_1 := A \cap B, A_2 := A \cap B^c.$

We now take $S \in \mathcal{G}(Z)$ with the following properties: if $z \in Z \backslash A, Sz := Tz, SA_1 = TA_2,$ and $SA_2=TA_1.$ Note that because $T$ is an extension and $A$ is not a cylinder set, $S \notin \mathcal{G}_X.$ Further note that $\{z \in Z : Sz \neq Tz  \}=A.$ Therefore, 

$$\sup_E \mu(TE \triangle SE) = d(T,S) \le d'(T,S) = \mu(A) < \epsilon.$$

So $S \in N_{\epsilon}(T).$  
\end{proof}

By Proposition \ref{Prop:GXIsClosed}, $\mathcal{G}_X$ is a closed subset of a Baire space, so $\mathcal{G}_X$ is itself a Baire space. Further, because $\mathcal{G}_X$ is nowhere dense, the classical Halmos and Rokhlin results can provide no information about $\mathcal{G}_X.$

\section{Discrete Extensions} \label{Sec:Discrete}
As stated in Section 2, throughout this paper we will let $(X,m)$ be the unit interval with Lebesgue measure. For this section, let $Z=X \times \{1,\ldots, L\},$ with $L \ge 2$, $w$ be a probability measure on $\{1, \ldots, L \}$, and $w_i := w(i)$ (without loss of generality, $w_i \neq 0$ for all $i$). Let $\mu$ be the product measure of $m$ and $w$ on $Z$.

In this section we will be exploring some results regarding these discrete extension measure spaces. We begin by showing that such systems can never be weakly mixing extensions.

\begin{prop} \label{Prop:NoWeaklyMixingFinite}
Let $(Z,\mu), (X,m)$ be as above. Then $\mathcal{W}_X = \emptyset.$
\end{prop}

\begin{proof}
Fix $T \in \mathcal{G}_X.$ It suffices to show that there exists an $f \in L^2(Z|X)$ with relative mean zero (that is, $\bbe(f|X)=0$ $m-$almost everywhere) such that 

$$\lim_{N \to \infty}\frac{1}{N} \sum_{n=0}^{N-1}{ \left( \int_X \abs{\bbe(T^nf \overline{f}|X)}^2 dm \right)^{1/2} } \neq 0.$$ 

In particular, we will construct $f$ such that $\bbe(T^nf \overline{f}|X)(x)$ can take only a finite number of possible values, none of which are 0. Thus $\abs{\bbe(T^nf \overline{f}|X)}^2(x)$ is always positive, $\frac{1}{N} \sum_{n=0}^{N-1}{\abs{\bbe(T^nf \overline{f}|X)}^2(x)}$ is bounded away from 0, and $\frac{1}{N} \sum_{n=0}^{N-1}{ \left( \int_X \abs{\bbe(T^nf \overline{f}|X)}^2 dm \right)^{1/2} }$ cannot converge to the zero function on $X$. 

Consider $f(x,y)$, where $f(x,i) = 1$ for all $x \in X, i = 2, \ldots, L$, and 

$$f(x,1)=\frac{-\sum_{i=2}^{L}{w_i}}{w_1}$$

for all $x \in X$. It is easy to see that $f$ has relative mean zero when $L \ge 2$,which is why we made this assumption at the beginning of the section.

Let $\sigma_{n,x}(i)$ be such that $T^n(x,i)=((T')^nx,\sigma_{n,x}(i))$ for all $(x,i) \in Z$. Now, 

$$\bbe(T^nf \overline{f}|X)(x) = \sum_{j=1}^{L}{w_j f(x,j) \fatn{j}} = \sum_{j=1}^{L}{w_j f(x,j) f(x, \sigma_{n,x}(j))},$$

with the last equality because $f$ is constant on any given level. Thus we see that because $T$ is invertible, $\sigma_{n,x}$ is a permutation on an $L$ element set, and the value of $\bbe(T^nf \overline{f}|X)(x)$ is completely determined by the specific permutation $\sigma_{n,x}$. As there are $L!$ permutations of the $L$ levels, there are finitely many possible values of $\bbe(T^nf \overline{f}|X)(x)$. 

To see $\bbe(T^nf \overline{f}|X)(x) \neq 0$, consider 2 cases. In the first case, we have $\sigma_{n,x}(1)=1$. In this case it is easy to see that every summand of $\sum_{j=1}^{L}{w_j f(x,j) f(x, \sigma_{n,x}(j))}$ is positive, and thus the sum is positive (in particular, nonzero). So now suppose $\sigma_{n,x}(i)=1, i \neq 1$. In this case we have

$$\sum_{j=1}^{L}{w_j f(x,j) f(x, \sigma_{n,x}(j))}= f(x,1)(w_1+w_i) + \left(\sum_{j=2}^{L}{w_j}\right)-w_i.$$

Consider 

$$f(x,1)(w_1+w_i) = \frac{-\sum_{j=2}^{L}w_j}{w_1} (w_1 + w_i) = \left(-\sum_{j=2}^{L}w_j\right) \left(1 + \frac{w_i}{w_1}\right).$$ 

Note that $\left(-\sum_{j=2}^{L}w_j\right) \left(1 + \frac{w_i}{w_1}\right) \le -\sum_{j=2}^{L}w_j$. Thus,

$$f(x,1)(w_1+w_i) + \left(\sum_{j=2}^{L}{w_j}\right)-w_i \le \left(-\sum_{j=2}^{L}w_j\right) + \left(\sum_{j=2}^{L}{w_j}\right)-w_i = - w_i < 0.$$ So $\bbe(T^nf \overline{f}|X)(x)$ is always nonzero, and $\abs{\bbe(T^nf \overline{f}|X)}^2(x)$ is always positive, as desired.
\end{proof}

We make two notes here. First, the proof of Proposition \ref{Prop:NoWeaklyMixingFinite} never used any assumptions on the factor, $(X,m,T'),$ and thus it will hold when the factor is any probability space, with any measure preserving transformation on that space. Second, the proof is still valid in the case that $Z$ has countably many levels instead of finitely many. The key observation is that for almost all $z$, if $z, Tz$ are on levels $k_1, k_2$ respectively, then $w(k_1)=w(k_2).$ As for any fixed $\alpha \in (0,1)$, there can be only finitely many levels $k$ with $w(k)=\alpha, T$ decomposes into invariant subsystems, to each of which we can apply Proposition \ref{Prop:NoWeaklyMixingFinite}. 

Though $\mathcal{W}_X$ is empty on these discrete extension spaces, they are still worth exploring. But before we can proceed, we will henceforth suppose that the probability measure $w$ is the normalized counting measure. That is, $w_i= \frac{1}{L}$ for all $i$. With this assumption, we extend the notion of dyadic sets and permutations on $X$ to  dyadic sets and permutations on $Z.$

\begin{defin} \label{Def:DyadicFinite} 
	If $D$ is a dyadic interval of rank $k$ in $X$, then a dyadic square of rank $k$ in $Z$ is a set of the form $D \times \{i\}$. A dyadic set in $Z$ is a union of dyadic squares. A dyadic permutation of rank $k$ on $Z$ is a permutation of the dyadic squares of rank $k$. A \textit{column-preserving (dyadic) permutation} (of rank $k$) on $Z$ is a dyadic permutation on $Z$ which is an extension of a dyadic permutation on $X$.
\end{defin}

We wish to generalize the fact that dyadic permutations are dense in $\mathcal{G}(X)$ to density of column-preserving permutations in $\mathcal{G}_X$. To this end, we make a couple notes. First we introduce the following notation: we write $A \subset i$ if there exists $A' \subset X$ such that $A = A' \times \{i\}$. 

Second, we will require the use of the following lemma by Halmos (for proof, see \cite[p.67]{HalmosLectures}).

\begin{lem} \label{Lem:DyadicSetsPartition}
Let $\{E_i : i = 1, \ldots n \}$ partition the unit interval, and $r_i$ be dyadic rationals such that $\sum_{i=1}^n r_i = 1$ and $\abs{m(E_i)-r_i} < \delta$ for some $\delta > 0$ and for all $i$. Then there exists $\{F_i : i = 1, \ldots n \}$, dyadic sets that partition the unit interval such that $m(F_i)=r_i$ and $m(E_i \triangle F_i) < 2 \delta$ for all $i$.
\end{lem}

We now move to the main result of this section.

%\begin{remark} \label{Remark:PartitionDyadicInterval}
%It is easy to see that the lemma also holds on a dyadic interval $D$ if we replace the summing condition on the $r_i$ with one $\sum_{i=1}^n r_i = \mu(D)$.
%\end{remark}  

\begin{thm}[Density of column-preserving permutations] \label{Thm:DensityOfPermutationsFinite} 
Column-preserving permutations are dense in $\mathcal{G}_X.$ More precisely, let $T \in \mathcal{G}_X.$  Given $N_{\epsilon}(T)$, a dyadic neighborhood of $T$, there exists $Q \in N_{\epsilon}(T)$, a column-preserving permutation.
\end{thm}

\begin{proof}
Without loss of generality, assume 

$$N_{\epsilon}(T)=\{S \in \mathcal{G}_X: \mu(TD_l \triangle SD_l)<\epsilon, l = 1, \ldots, L(2^n) \},$$ where $D_l$ are every dyadic square of some fixed rank $n$ (note that $D_{l_1}, D_{l_2}$ are disjoint up to boundary points). 

Let $k \in \{1, \ldots, L\}$, and let $P_k := \{D_i \cap TD_j|D_i \subset k, j = 1,\ldots, L(2^n)\}$. Note that $P_k$ partitions level $k$. If $\pi$ is the natural projection onto $X$, then let $P'_k := \pi P_k = \{\pi E| E \in P_k\}$. $P'_k$ is a partition of $X$. Let $P' = \{\hat{A}_{\lambda}| \lambda \in \Lambda \}$ be a common refinement of $P'_k$ for $k = 1, \ldots, L$, and let $P = \{\hat{A}_{\lambda,k} | \lambda \in \Lambda, k = 1, \ldots, L\}$ be a partition of $Z$ obtained by lifting every element of $P'$ to every level ($\hat{A}_{\lambda,k} \subset k$). 

Applying a weaker version of Lemma \ref{Lem:DyadicSetsPartition} (one where we do not care about the value of $\abs{m(E_i)-r_i}$ in the formulation of the lemma) to the partition $P'$, we obtain a partition $\{A_{\lambda}\}$ of $X$ into dyadic sets so that $m(\hat{A}_{\lambda} \triangle A_{\lambda})< \frac{L\epsilon}{2 \abs{\Lambda}}$. Applying Lemma \ref{Lem:DyadicSetsPartition} again, we get a partition of $X$ into dyadic sets $B_{\lambda}$ so that $m((T')^{-1}\hat{A}_{\lambda} \triangle B_{\lambda})< \frac{L \epsilon}{2 \abs{\Lambda}}$. Note the full strength of Lemma \ref{Lem:DyadicSetsPartition} guarantees we can select this partition so that $m(A_{\lambda})=m(B_{\lambda})$ (as $m(\hat{A}_{\lambda})=m((T')^{-1}\hat{A}_{\lambda})$). We can now lift $A_{\lambda},B_{\lambda}$ to sets $A_{\lambda,k},B_{\lambda,k}$ so that $A_{\lambda,k},B_{\lambda,k} \subset k$. Note that

 \begin{equation} \label{Eq:DiscreteApproximations}
 \mu(\hat{A}_{\lambda,k} \triangle A_{\lambda,k})< \frac{\epsilon}{2  \abs{\Lambda}}, \text{ and } \mu(T^{-1}\hat{A}_{\lambda,k_2} \triangle B_{\lambda,k_1})< \frac{\epsilon}{2 \abs{\Lambda}}.
 \end{equation}
 
where $k_1, k_2$ are such that if $i,j$ are such that $\hat{A}_{\lambda,k_2} \subset D_i \cap TD_j,$ then $D_i \subset k_2, D_j \subset k_1$. 

We will now define $Q$ of some rank $r \in \mathbb{N}$ where $r$ is at least as large as the ranks of $D_i$ for every $i$, and $A_{\lambda},B_{\lambda}$ for every $\lambda$. We first define $Q'$ a dyadic permutation on $X$ as any dyadic permutation which maps $B_{\lambda}$ to $A_{\lambda}$ for every $\lambda$. Next we define $Q$, a column preserving permutation of rank $r$. First let $k_1,k_2$ be as before: if $i,j$ are such that $\hat{A}_{\lambda,k_2} \subset D_i \cap TD_j$, then $D_i \subset k_2, D_j \subset k_1$. Then $Q$ will be the extension of $Q'$ such that $B_{\lambda,k_1} \mapsto A_{\lambda,k_2}$. Note that for all $\lambda, k,$ we have that $\text{level } Q^{-1} A_{\lambda,k} = \text{level } T^{-1} \hat{A}_{\lambda,k},$ where $\text{level }A := k$ if and only if $A \subset k.$

We now show that $\mu(TD_j \triangle QD_j) < \epsilon$ for all $j$. Fix $j \in 1, \ldots, L(2^n),$ and define $k$ so that $D_j \subset k$. Let $\Lambda_j := \{\lambda \in \Lambda| (T')^{-1}\hat{A}_{\lambda} \subset \pi D_j \}$. For $\lambda \in \Lambda_j,$ let $i_{\lambda,j}$ be such that $T^{-1}\hat{A}_{\lambda, i_{\lambda,j}} \subset D_j$. Then $D_j = \bigcup_{\lambda \in \Lambda_j}{T^{-1}\hat{A}_{\lambda, i_{\lambda,j}}}.$

Further, by the definitions of $Q$ and $\Lambda_j,$ as well as the previous note, we have that $Q^{-1}A_{\lambda, i_{\lambda,j}} = B_{\lambda,k}.$ 

Note all unions and sums will be taken over $\lambda \in \Lambda_j.$ We have

 \begin{equation} \label{Eq:DiscreteFinal1}
 \mu\left(D_j \triangle \bigcup B_{\lambda,k}\right) = \mu \left(\bigcup T^{-1} \hat{A}_{\lambda,i_{\lambda,j}} \triangle \bigcup B_{\lambda,k}\right) \le \sum \mu(T^{-1}\hat{A}_{\lambda,i_{\lambda,j}} \triangle B_{\lambda,k}).
 \end{equation}
 
But by (\ref{Eq:DiscreteApproximations}), $\mu(T^{-1}\hat{A}_{\lambda,i_{\lambda,j}} \triangle B_{\lambda,k}) < \frac{\epsilon}{2 \abs{\Lambda}}$ so $(\ref{Eq:DiscreteFinal1}) < \sum \frac{\epsilon}{2 \abs{\Lambda}} = \frac{\epsilon}{2}.$ Therefore

 $$\mu\left(QD_j \triangle \bigcup A_{\lambda,i_{\lambda,j}}\right) = \mu\left(D_j \triangle \bigcup B_{\lambda,k}\right) \le \frac{\epsilon}{2}.$$

On the other hand, 

\begin{equation} \label{Eq:DiscreteFinal2}
\mu\left(\bigcup A_{\lambda,i_{\lambda,j}} \triangle TD_j\right) = \mu\left(\bigcup A_{\lambda,i_{\lambda,j}} \triangle \bigcup \hat{A}_{\lambda,i_{\lambda,j}}\right) \le \sum (A_{\lambda,i_{\lambda,j}} \triangle \hat{A}_{\lambda,i_{\lambda,j}}).
\end{equation}
 
 Again by (\ref{Eq:DiscreteApproximations}), $\,u(A_{\lambda,i_{\lambda,j}} \triangle \hat{A}_{\lambda,i_{\lambda,j}}) < \frac{\epsilon}{2 \abs{\Lambda}}$ so $ (\ref{Eq:DiscreteFinal2}) < \sum \frac{\epsilon}{2 \abs{\Lambda}} = \frac{\epsilon}{2}$. Finally,  $$\mu(TD_j \triangle QD_j) \le \mu\left(TD_j \triangle \bigcup A_{\lambda,i_{\lambda,j}} \right) + \left(\bigcup A_{\lambda,i_{\lambda,j}} \triangle  QD_j\right) \le \frac{\epsilon}{2} + \frac{\epsilon}{2} = \epsilon.$$ And because this holds for all $j,$ we have that $Q \in N_{\epsilon}(T).$ 
\end{proof}

\section{Weak Approximation Theorem for Extensions on the Unit Square} \label{Sec:WAT}
Now we let $(Z,m_2)$ be $X \times X$ with the Lebesgue measure. If we need further clarity, we will write the Lebesgue measure on $X$ as $m_1$, but in general we will denote both Lebesgue measures by $m$.

We begin by drawing some connections to Section \ref{Sec:Discrete}. First, however, we need some more notation. For $L \in \mathbb{N},$ define $Z_L:= \bigcup_{j=0}^{L-1}{\left(X \times \left \{\frac{j}{L}\right \}\right)} \subset Z, \mu_L,$ a measure on $Z_L,$ to be the product of the Lebesgue measure with a normalized counting measure on $L$ points. Further, $\pi_L: Z \to Z_L$ be the natural projection onto $Z_L.$ That is, if $z = \left(x, \frac{j}{L} + \gamma \right)$ for $\gamma \in \left[0, \frac{1}{L} \right),$ then $\pi_L (z) = \left(x, \frac{j}{L} \right).$ 

\begin{defin} \label{Def:DiscreteEquivalent}
Let $T \in \mathcal{G}(Z)$. We say that $T$ is \textit{discrete equivalent} if there exists $L$ and $T_L \in \mathcal{G}(Z_L)$, such that $(Z,m,T)$ is an extension of $(Z_L,\mu_L,T_L)$ through the factor map $\pi_L.$ Further, we say that $T$ is \textit{simply discrete equivalent} if $T$ is an identity extension. That is, if we write $Z$ as $Z_L \times \left[0, \frac{1}{L} \right),$ then $T=T_L \times I.$ If we wish to emphasize the number of levels, $L$, we will say $T$ is \textit{$L$-(simply) discrete equivalent}. 
\end{defin}

Definition \ref{Def:DiscreteEquivalent} is fairly easy to visualize. We take the square and divide it into $L$ equal measure horizontal pieces. Then $T$ is discrete equivalent if $T$ moves fibers on each small piece to other such fibers, and is simply discrete equivalent if it does not move any points within the fiber. Note that in general, a discrete equivalent $T$ need not be in $\mathcal{G}_X.$ However, if $T \in \mathcal{G}_X,$ then $T_L$ is also an extension of $T'.$

Our goal for this section is to provide a version of Halmos' Weak Approximation Theorem (see \cite[p.65]{HalmosLectures}) when restricted to $\mathcal{G}_X$. Mostly this will mean proving a result equivalent to Theorem \ref{Thm:DensityOfPermutationsFinite}. However we first need to lay some ground work. Definitions of dyadic squares, sets, and permutations are all standard in this case, so we do not redefine them. Column-preserving permutations are defined just as they are in Definition \ref{Def:DyadicFinite}. 

Before moving on, we make a few remarks.

\begin{remark}
Lemma \ref{Lem:DyadicSetsPartition} holds on $(Z,m),$ because $(Z,m),$ like $(X,m),$ is a non-atomic standard probability space, and Lemma \ref{Lem:DyadicSetsPartition} holds for all such spaces (replacing ``dyadic sets'' in the statement of Lemma \ref{Lem:DyadicSetsPartition} with a class $\mathcal{B}$ which is isomorphic to the class of dyadic sets). Alternatively, one can simply prove Lemma \ref{Lem:DyadicSetsPartition} again in the context of the square. No part of the proof relies on the the fact that we were on the unit interval, so nothing changes in the proof.
\end{remark}

\begin{remark}
We will make the following notational convenience. If $T \in \mathcal{G}(Z)$ and $S' \in \mathcal{G}(X),$ then we will write $S'T$ in place of $(S' \times I) T.$
\end{remark}

\begin{remark} \label{Remark:PermutationsDiscreteEquivalent}
If $Q \in \mathcal{G}(Z)$ is a dyadic permutation of rank $K$, then $Q$ is $L$-simply discrete equivalent with $L = 2^K.$ Further if $S' \in \mathcal{G}(X),$ then $S'Q$ is also $L$-simply discrete equivalent. If $Q$ is further an extension of $Q' \in \mathcal{G}(X),$ then $S'Q$ is an extension of $S'Q'.$
\end{remark}

The key to our goal is the following strengthening of Lemma \ref{Lem:DyadicSetsPartition}.

\begin{lem} \label{Lem:ColumnPartitions}
	Let $\{E_1,\ldots E_N \}$ be a finite partition of $Z$, $\epsilon > 0,$  and suppose $\{\tilde{F}_1,\ldots,\tilde{F}_N \}$ is another partition of $Z,$ where $\tilde{F}_i$ are all dyadic sets, and $m(E_i \triangle \tilde{F}_i) < \epsilon.$ Let $K := \text{\emph{max rank} } \tilde{F}_i$ and let $E_{ij} := E_i \cap \pi^{-1}C_j, j \in \{1, \ldots, 2^K \},$ where $C_j$ is a dyadic interval of rank $K$. Let $r_{ij}$ be dyadic rationals (possibly zero) such that $\sum_{i=1}^N r_{ij} = \frac{1}{2^K}$ for all $j$ and $\abs{m(E_{ij})-r_{ij}} < \frac{\epsilon}{2^{K}}$ for all $i,j.$ Then there exists $\{F_1, \ldots, F_N\}$ a partition of $Z$ such that $F_{i}$ is dyadic set for all $i, m(F_{ij}) = r_{ij}$ (with $F_{ij}$ similarly defined as $E_{ij}$) for all $i,j$, and $m(E_{i} \triangle {F_{i}}) < 3 \epsilon$ for all $i.$
\end{lem}

\begin{proof}
	Similar to the definition of $E_{ij},$ define $\tilde{F}_{ij} := \tilde{F}_i \cap \pi^{-1}C_j.$ Note that by choice of $K, \tilde{F}_{ij}$ is of the product of $C_j$ and a dyadic set for all $i,j$. Now, for all $\tilde{F}_{ij}$ with $m(\tilde{F}_{ij}) > r_{ij},$ let $A_{ij} \subset \tilde{F}_{ij}$ of the form $A_{ij} = C_j \times B_{ij}$ with $B_{ij}$ a dyadic set, and 	$m(A_{ij}) = m(\tilde{F}_{ij}) - r_{ij}.$ Define $F_{ij} := \tilde{F}_{ij} \backslash A_{ij}.$ Let $A$ be the union of all $A_{ij}$ chosen up to this point. Now for $\tilde{F}_{ij}$ with $m(\tilde{F}_{ij}) < r_{ij},$ let $A_{ij} \subset A$ of the same form as above, this time with $m(A_{ij}) = r_{ij} - m(\tilde{F}_{ij}).$ In this case, define $F_{ij} := \tilde{F}_{ij} \cup A_{ij}.$ Now let $F_i := \bigcup_{j=1}^{2^K} F_{ij}.$ Note that some $F_{ij}$ may be empty. In particular, $F_{ij} = \emptyset$ if and only if  $r_{ij}=0$. 
	
	Note that by definition, $m(F_{ij}) = r_{ij}$ and note further that $\tilde{F}_{ij} \triangle F_{ij} = A_{ij}.$ We claim 
	
	$$\sum_{j=1}^{2^K} m(A_{ij}) < 2 \epsilon$$ 
	
	for all $i$. Let $i$ be fixed, and consider 
	
	$$\sum_j m(A_{ij}) = \sum_j \abs{m(\tilde{F}_{ij}) - r_{ij}} \le \sum_j \abs{m(E_{ij}) - r_{ij}} + \sum_j \abs{m(E_{ij}) - m(\tilde{F}_{ij})}.$$ 
	
	We have $\abs{m(E_{ij})-r_{ij}} < \frac{\epsilon}{2^K}$ so $\sum_j \abs{m(E_{ij}) - r_{ij}} < \epsilon.$ On the other hand, 
	
	$$\sum_j \abs{m(E_{ij}) - m(\tilde{F}_{ij})} \le \sum_j m(E_{ij} \triangle \tilde{F}_{ij}) = m\left( \bigcup_j (E_{ij} \triangle \tilde{F}_{ij})\right).$$
	
	 Now, because $m(E_{ij_1} \cap \tilde{F}_{ij_2}) = 0$ if $j_1 \neq j_2,$ we have that 
	 
	 $$m\left( \bigcup_j (E_{ij} \triangle \tilde{F}_{ij})\right) = m \left(\bigcup_j E_{ij} \triangle \bigcup_j \tilde{F}_{ij} \right) = m(E_i \triangle \tilde{F}_i) < \epsilon.$$
	 
	  Therefore, $\sum_j m(A_{ij}) < 2 \epsilon.$ 
	
	We will now show $m(E_i \triangle F_i) < 3 \epsilon.$ Firstly, we have $m(E_i \triangle F_i) \le m(E_i \triangle \tilde{F}_i) + m(\tilde{F}_i \triangle F_i).$ But $m(E_i \triangle \tilde{F}_i) < \epsilon.$ Further, 
	
	$$m(\tilde{F}_i \triangle F_i) = m\left( \bigcup_j \tilde{F}_{ij} \triangle \bigcup_j F_{ij} \right) \le m\left( \bigcup_j (\tilde{F}_{ij} \triangle F_{ij}) \right) = \sum_j m(\tilde{F}_{ij} \triangle F_{ij}).$$
	
	 But as previously noted, $\tilde{F}_{ij} \triangle F_{ij} = A_{ij},$ and we already showed $\sum_j m(A_{ij}) < 2 \epsilon$. Thus, $m(E_i \triangle F_i) < 3 \epsilon$ as desired.
\end{proof}

With Lemma \ref{Lem:ColumnPartitions}, we can now prove the equivalent version of Theorem \ref{Thm:DensityOfPermutationsFinite} for the unit square, which will be the core result for proving our version of the Weak Approximation Theorem.

\begin{thm}[Density of column-preserving permutations] \label{Thm:DensityOfPermutations}
Column-preserving permutations are dense in $\mathcal{G}_X.$ More precisely, let $T \in \mathcal{G}_X.$  Given $N_{\epsilon}(T)$, a dyadic neighborhood of $T$, there exists $Q \in N_{\epsilon}(T)$, a column-preserving permutation.
\end{thm}

\begin{proof}
We may assume without loss of generality that 

$$N_{\epsilon}(T) = \{S: \mu(TD_i \triangle SD_i)<\epsilon, l = 1, \ldots, (2^{2N}) \},$$

where $D_i$ are dyadic squares of some fixed rank $N$. We start with the case where $T'=I_X.$

Let $D_{ij} := D_i \cap TD_j.$ Note that $\{D_{ij}\}$ partitions $Z$. By Lemma \ref{Lem:DyadicSetsPartition}, there exists a partition of $Z$ into dyadic sets, $\{\tilde{E}_{ij}\},$ such that $m(D_{ij} \triangle \tilde{E}_{ij}) < \frac{\epsilon}{6M},$ where $M := 2^{2N}.$ Further by Lemma \ref{Lem:DyadicSetsPartition}, we can find a dyadic partition of $Z$ into sets $\{\tilde{F}_{ij}\}$ where $m(T^{-1}D_{ij} \triangle \tilde{F}_{ij}) < \frac{\epsilon}{6M}.$ Note that because $m(D_{ij}) = m(T^{-1}D_{ij}),$ we can assume that $m(\tilde{E}_{ij}) = m(\tilde{F}_{ij})$ Let $K = \text{max rank} \{\tilde{E}_{ij}, \tilde{F}_{ij}\}.$ We can now apply Lemma \ref{Lem:ColumnPartitions} to both $\tilde{E}_{ij}$ and $\tilde{F}_{ij}$ to get dyadic partitions $\{E_{ij}\}$ and $\{F_{ij}\}$ such that 

\begin{equation} \label{Eq:SquareApproximations}
m(D_{ij} \triangle E_{ij}) < \frac{\epsilon}{2M} \text{ and } m(T^{-1}D_{ij} \triangle F_{ij}) < \frac{\epsilon}{2M}.
\end{equation}

Recall that if $C_k$ is a dyadic interval of rank $K$, then in the notation of Lemma \ref{Lem:ColumnPartitions}, $E_{ijk} := E_{ij} \cap \pi^{-1}C_k$ and $F_{ijk} := F_{ij} \cap \pi^{-1}C_k.$ Note that not only do we have $m(D_{ij}) = m(T^{-1}D_{ij}),$ but because $T$ is an extension of identity, $m(T^{-1}D_{ij} \cap \pi^{-1}C_k) = m(T^{-1}(D_{ij} \cap \pi^{-1}C_k)) = m(D_{ij} \cap \pi^{-1}C_k).$ Thus we are able to choose the same dyadic rationals in both applications of Lemma \ref{Lem:ColumnPartitions}, and subsequently have that $m(E_{ijk}) = m(F_{ijk})$ for $i,j = 1, \ldots 2^{2N}, k = 1, \ldots, 2^K.$

We now define $Q$ as the permutation which maps $F_{ijk}$ to $E_{ijk}.$ Note that in particular, $Q$ will map $F_{ij}$ to $E_{ij}.$ Further note that $Q$ will be an extension of the identity. 

Let $j$ be fixed. We will now show $m(QD_j \triangle TD_j) < \epsilon$. Recall $D_{ij}= D_i \cap TD_j,$ so $T^{-1}D_{ij} = T^{-1}D_i \cap D_j$ and $D_j = \bigcup_i T^{-1}D_{ij}.$ We have

\begin{equation} \label{Eq:SquareFinal1}
 m\left(D_j \triangle \bigcup_i F_{ij} \right) = m \left( \bigcup_i T^{-1} D_{ij} \triangle \bigcup_i F_{ij} \right) \le \sum_i m(T^{-1} D_{ij} \triangle F_{ij}).
\end{equation}

But per (\ref{Eq:SquareApproximations}), $m(T^{-1}D_{ij} \triangle F_{ij}) < \frac{\epsilon}{2M},$ so $(\ref{Eq:SquareFinal1}) < \sum_i \frac{\epsilon}{2M} \le \frac{\epsilon}{2}$. Therefore

$$
m\left(QD_j \triangle \bigcup_i E_{ij} \right) = m \left(D_j \triangle \bigcup_i F_{ij} \right) < \frac{\epsilon}{2}.
$$

On the other hand, 

\begin{equation} \label{Eq:SquareFinal2}
m\left( TD_j \triangle \bigcup_i E_{ij}  \right) = m \left( \bigcup_i D_{ij} \triangle \bigcup_i E_{ij} \right) \le \sum_i m(D_{ij} \triangle E_{ij}).
\end{equation}

Again, per (\ref{Eq:SquareApproximations}), $m(D_{ij} \triangle E_{ij}) < \frac{\epsilon}{2M},$ so $(\ref{Eq:SquareFinal2}) < \sum_i \frac{\epsilon}{2M} = \frac{\epsilon}{2}.$ Therefore,

$$
m(TD_j \triangle QD_j) \le m\left( TD_j \triangle \bigcup_i E_{ij} \right) + m\left(\bigcup_i E_{ij} \triangle QD_j \right) < \frac{\epsilon}{2} + \frac{\epsilon}{2} = \epsilon.
$$

As this holds for all $j,$ we have that $Q \in N_{\epsilon}(T).$ 

Now suppose $T$ is an extension of some invertible $T'.$ Define $\tilde{T}:= (T')^{-1}T.$ Then $\tilde{T}$ is an extension of the identity, so there exists a column-preserving permutation $\tilde{Q} \in N_{\epsilon/2}(\tilde{T}).$ But then $T'\tilde{Q} \in N_{\epsilon/2}(T)$ as $m(T'\tilde{Q}D_i \triangle TD_i) = m(\tilde{Q}D_i \triangle \tilde{T}D_i) < \frac{\epsilon}{2}.$ By Remark \ref{Remark:PermutationsDiscreteEquivalent}, $T'\tilde{Q}$ is $L$-simply discrete equivalent, with $L = 2^{\text{rank } \tilde{Q}}$. If we let $G_i:= \pi_L D_i$ and let $\tilde{N}_{\epsilon/2}(\pi_L(T'\tilde{Q})) := \{S_L: \mu_L(\pi_L(T'\tilde{Q})G_i \triangle S_LG_i) < \frac{\epsilon}{2} \forall i \},$ then by Theorem \ref{Thm:DensityOfPermutationsFinite} there exists a column-preserving dyadic permutation $\hat{Q} \in \tilde{N}_{\epsilon/2}(\pi_L(T'\tilde{Q})).$ Now we define $Q$ to be the simply discrete equivalent extension of $\hat{Q}.$ Note that because $L$ was dyadic, $Q$ is a (column-preserving) dyadic permutation. Further, $Q \in N_{\epsilon/2}(T'\tilde{Q})$ as $m(QD_i \triangle T'\tilde{Q}D_i) = \mu_L(\hat{Q}G_i \triangle \pi_L(T'\tilde{Q})G_i) < \frac{\epsilon}{2}.$ So 

$$m(QD_i \triangle TD_i) \le m(QD_i \triangle T'\tilde{Q}D_i) + m(T'\tilde{Q}D_i \triangle TD_i) < \epsilon$$

for all $i$, and thus $Q \in N_{\epsilon}(T).$ 
\end{proof}

We close this section with the promised version of Halmos' Weak Approximation Theorem for extensions.

\begin{thm}[Weak Approximation Theorem for Extensions] \label{Thm:WATE}
	Let $T\in \mathcal{G}_X$, and let $N_{\epsilon}(T)$ be a dyadic neighborhood of $T.$ Then for any $k_0 \in \mathbb{N},$ there exists $k \ge k_0$ and $Q \in \mathcal{G}_X$such that the following hold:
	
	\begin{itemize}
		\item $Q,Q'$ are dyadic permutations of rank $k$ on $Z,X$ respectively,
		\item $Q'$ is cyclic,
		\item $Q$ is periodic with period $2^k$ everywhere, 
		\item $Q \in N_{\epsilon}(T).$
	\end{itemize} 
\end{thm}

\begin{proof}
Because Theorem \ref{Thm:DensityOfPermutations} tells us that $N_{\epsilon/2}(T)$ will contain $P \in \mathcal{G}_X,$ ($P$ a column-preserving permutation) we need only prove the case where $T$ is a permutation itself (we will therefore proceed using $P, P'$ in place of $T,T'$). 

Fix $k_0 \in \mathbb{N}.$ Because $P$ is a permutation and $D_i$ is a dyadic set, $PD_i$ is also a dyadic set. Let $M$ be the maximum rank of $PD_i$ (so that $P$ is a permutation of rank $M$), $K$ be the number of disjoint cycles in $P',$ and $k$ be chosen to be greater than both $M$ and $k_0,$ and such that $\frac{K}{2^{k-1}} < \epsilon$. 

We will now construct $Q$ of rank $k.$ Note that following the proof there will be an example of this construction. To start, let $E_1$ be any dyadic square of rank $M$. If $\pi E_1$ is not a fixed point of $P',$ we have $Q$ map the ``first'' rank $k$ dyadic square (which we will henceforth refer to as a $k$-square) of $E_1$ to the ``first'' $k$-square in $PE_1.$ By ``first'' $k$-square, we mean the top left $k$-square. Now, if $(P')^2 \pi E_1 \neq \pi E_1,$ we continue to map to the first $k$-square in $(P')^2 \pi E_1.$ Eventually, however, we reach a point where $(P')^l \pi E_1 = \pi E_1.$ From where we are in $(P)^{l-1} E_1,$ we continue to map to the ``second'' $k$-square in $P^l E_1$ (by ``second'' we mean the one to the right of the first). Note that $P^l E_1 \neq E_1$ in general. 

We now repeat the entire process, replacing ``first'' for ``second,'' eventually ``third'' and so on, as well as replacing $E_1$ with $P^l E_1$. Eventually we will arrive at a $k$-square whose projection is at the far right of $(P')^{l-1} \pi E_1.$ At this point, we choose an $M$-square $E_2$ such that $\pi E_2$ is not in the $P'$ cycle of $\pi E_1$ (assuming such an $E_2$ exists). Then from our current position, we map to the first $k$-square of $E_2,$ and repeat the process. 

We continue on like this until we have exhausted every $P'$ cycle (including fixed points), at which point we return to the the first $k$-square of $E_1.$ Note that we have visited every $k$-column exactly once. We are not quite done yet, though. We now choose a $k$-square on the same column as the first $k$-square in $E_1,$ and we repeat the entire process. Now shifting to rows within the $M$-squares that correspond to our new choice of starting point. That is, in the original process, we were in the top row of every $M$-square, because our original $k$-square was in the top row. If our new $k$-square is in the 3rd row within its $M$-square, say, all our choices will be in the 3rd row of the respective $M$-squares. Repeating this process, we eventually define $Q$ for all $k$-squares.

We now find a bound for $m(PD_i \triangle QD_i).$ Note that by our construction the only points that can be in $PD_i \triangle QD_i$ come from $k$-squares in $D_i$ whose projections are in the last $k$-interval in each $P'$ cycle. Let $E_j$ be such a $k$-square. Then $m(PE_j \triangle QE_j) \le \frac{2}{2^{2k}} = \frac{1}{2^{2k-1}}.$ There are $2^k$ such $E_j$ per $k$-column, and there are $K$ such $k$-columns.  Thus, $m(PD_i \triangle QD_i) \le  \bigcup_j m(PE_j \triangle QE_j) \le \frac{K 2^k}{2^{2k-1}} = \frac{K}{2^{k-1}} < \epsilon$.  
\end{proof}

The construction in the proof of Theorem \ref{Thm:WATE} can be difficult to follow closely, so we provide an example of the construction. We first provide a $P$ which, in this case, will be of rank $2$. See Figure \ref{Fig:4x4} for reference on how we label the $2$-squares. Note that we will define $P$ using cycle decomposition notation. That is, if we write $R=(1 \ 2  \ 3),$ then we mean that the image under $R$ of the square labeled $1$ is the square labeled $2$. Similarly the image of ``$2$'' is ``$3$'' and the image of ``$3$'' is ``$1$''. Any squares not written explicitly in the decomposition are fixed points.  Now, we let $P := (1 \ 11 \ 5 \ 3)(13 \ 15)(9 \ 7)(2 \ 6 \ 14)(4 \ 16 \ 12 \ 8).$ Note that $P$ extends $P' := (1 \ 3)$ on $X$.

\begin{figure}[h] 
\caption{}\label{Fig:4x4}
\includegraphics[scale=0.4]{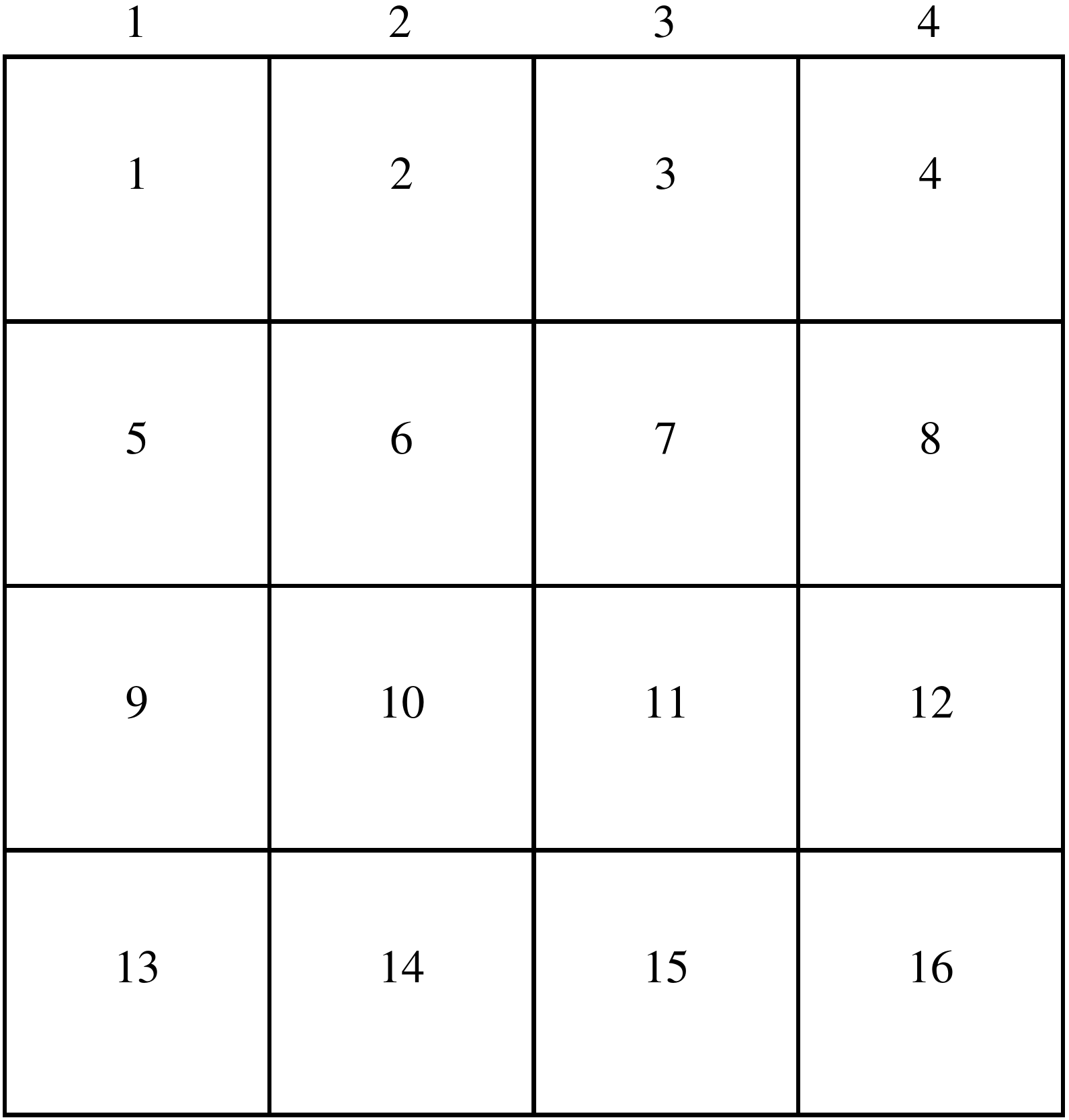}
\centering
\end{figure}

\begin{figure}[h] 
\caption{}\label{Fig:8x8}
\includegraphics[scale=0.4]{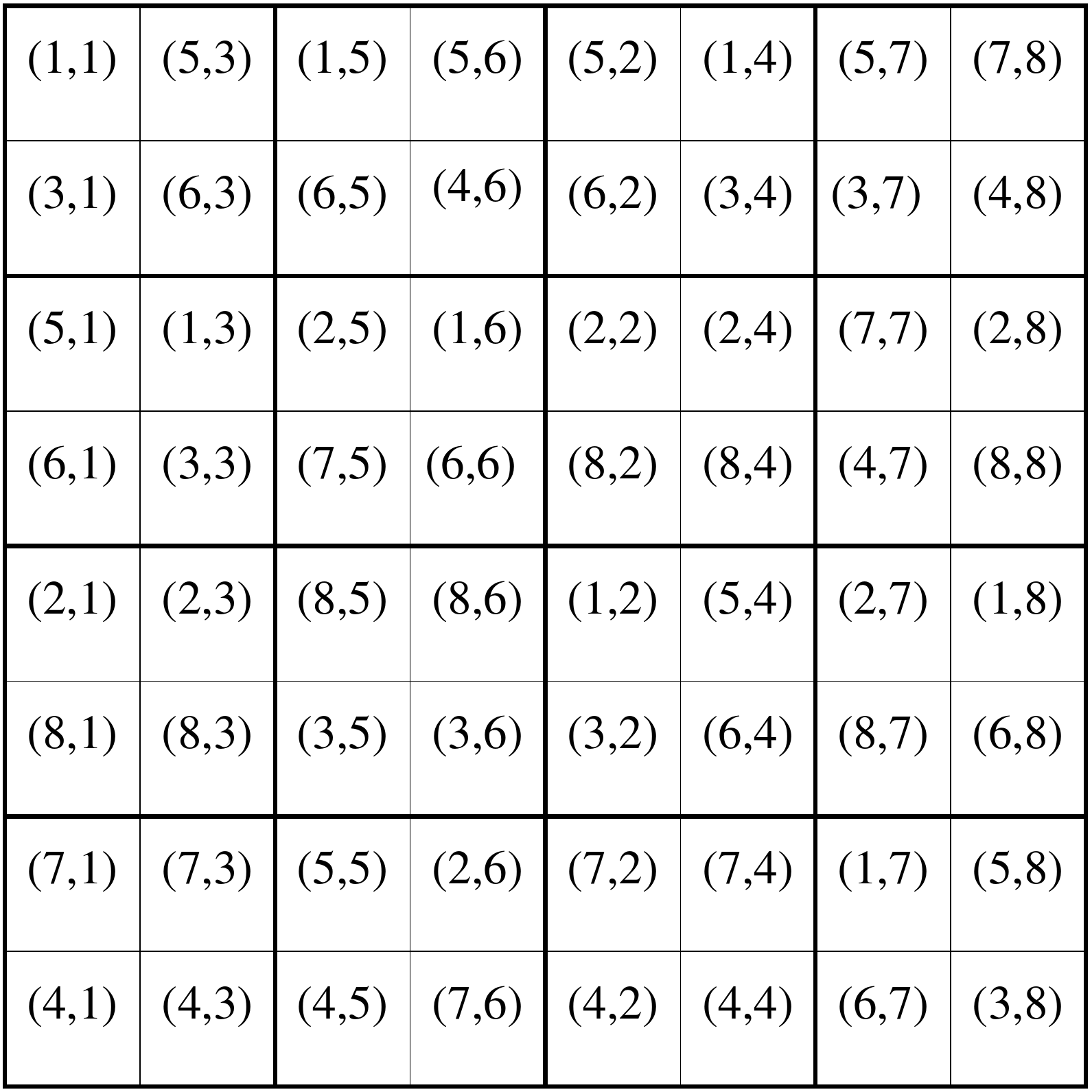}
\centering
\end{figure}

Suppose we were to construct $Q$ to be a rank $3$ permutation. Rather than write the entire cycle decomposition of $Q$ (as it would involve writing all $64$ $3$-squares), we label Figure \ref{Fig:8x8} to define $Q$. Here we have labeled the $3$-squares such that for a square labeled $(n,k),$ we have that $Q(n,k)=(n,k+1), (k \mod 8)$ (for consistency, here we have $8 \mod 8 := 8$ instead of $0$ as it typically would be). Further, if $n_1 \neq n_2,$ then $(n_1,k_1),(n_2,k_2)$ are in independent cycles. It is easy to see with this notation that $Q$ is an extension of a cyclic permutation $Q'$ on $X.$ We also note that the $Q$ we constructed is not the only possible $Q$ we could have constructed, as we have many free choices in the construction. 

To close this section, we note that a very simple modification of the proof of Theorem \ref{Thm:WATE} would yield a column preserving permutation $Q$ such that not only $Q'$ is cyclic, but $Q$ is cyclic as well. In our example seen in Figure \ref{Fig:8x8}, this modification would be accomplished by changing the definition of $Q$ slightly so that $Q(n,8)=(n+1,1),(n \mod 8)$. This formulation is more akin to the classical theorem. However, we choose the formulation given in Theorem \ref{Thm:WATE} as it is this formulation we need for further results.

\section{Uniform Approximation} \label{Sec:UniformApproximation}
Our goal in this section is to prove results that are generalizations of those needed for Halmos' classical Conjugacy Lemma (the key lemma for proving that weakly mixing transformations on $X$ are dense in $\mathcal{G}(X)$), and whose proofs quickly follow from the classical results and their proofs.

\begin{lem} \label{Lem:PeriodicExtensionPartition}
	Let $T \in \mathcal{G}_X$ where $T'$ is periodic of period $n$ (almost) everywhere. Then there exists a set $E$ such that $E= \pi^{-1}E'$ for some $E' \subset X,$ and $\{E,TE, \ldots, T^{n-1}E\}$ partition $Z.$ 
\end{lem}

\begin{proof}
	Because $T'$ is has period $n$ everywhere, there exists $E'$ such that $\{E',T'E',\ldots, (T')^{n-1}E'\}$ partitions $X$. Setting $E:=\pi^{-1}E'$ we have $$\{E,TE, \ldots, T^{n-1}E\}$$ are pairwise disjoint because $T$ extends $T'$. Further, because $m(E)=m(E')=\frac{1}{n},$ we have $m\left(\bigcup_{i=0}^{n-1} T^iE\right)= \sum_{i=0}^{n-1} m(T^iE) = 1,$ or $\bigcup_{i=0}^{n-1} T^iE = X$.
\end{proof}

Next we move to a version of Rokhlin's lemma (see, for example, \cite[p.71]{HalmosLectures}).

\begin{lem} \label{Lem:AntiperiodicExtension}
	Let $T \in \mathcal{G}_X$ where $T'$ is antiperiodic. Then for every $n \in \mathbb{N}$ and $\epsilon >0$ there exists $E$ such that $E=\pi^{-1}E'$ for some $E',\{E,TE, \ldots, T^{n-1}E\}$ are pairwise disjoint, and $m\left(\bigcup_{i=0}^{n-1} T^iE\right) > 1-\epsilon.$
\end{lem}

\begin{proof}
	Let $n \in \mathbb{N}$ and $\epsilon >0$. Because $T'$ is antiperiodic, there exists $E' \subset X$ such that $\{E',T'E',\ldots, (T')^{n-1}E'\}$ are pairwise disjoint and $m\left(\bigcup_{i=0}^{n-1} (T')^iE'\right) > 1-\epsilon$. Let $E:= \pi^{-1}E'$. Because $T$ extends $T',\{E,TE, \ldots, T^{n-1}E\}$ are pairwise disjoint. Further $m\left(\bigcup_{i=0}^{n-1} T^iE\right)= \sum_{i=0}^{n-1} m(T^iE)=\sum_{i=0}^{n-1} m((T')^iE') = m\left(\bigcup_{i=0}^{n-1} (T')^iE'\right) > 1-\epsilon.$
\end{proof}

We conclude this section with a version of Halmos' Uniform Approximation Theorem (see \cite[p.75]{HalmosLectures}). 

\begin{thm}[Uniform Approximation Theorem for Extensions] \label{Thm:UATE}
	Let $T \in \mathcal{G}_X$\ where $T'$ is antiperiodic. Then for every $n \in \mathbb{N}$ and $\epsilon >0$ there exists $R \in \mathcal{G}_X,$ such that both $R$ and $R'$ are periodic with period $n$ almost everywhere, and $d'(R,T) \le \frac{1}{n} + \epsilon.$
\end{thm}

\begin{proof}
	By Lemma \ref{Lem:AntiperiodicExtension}, there exists $E$ a cylinder set, such that $\{E,TE, \ldots, T^{n-1}E\}$ are pairwise disjoint, and $m\left(\bigcup_{i=0}^{n-1} T^iE\right) > 1-\epsilon$. If $z \in \bigcup_{i=0}^{n-2} T^iE,$ define $Rz := Tz,$ and if $z \in T^{n-1}E,$ define $Rz := T^{-(n-1)}z,$ thus making $R$ have period $n$ for all points on which we have thus far defined it. Further, because $T$ extends $T', R$ is also an extension. And for any definition of $R$ on the remainder of $Z$, we have $d'(R,T) \le m(T^{n-1}E) + \epsilon \le \frac{1}{n} + \epsilon.$ 
	
	All that remains is to define $R$ on the remainder of $Z$ so that $R$ is an extension, and $R,R'$ have period $n.$ Since the remainder is a cylinder set, this can be done by defining $R'$ on the projection of the remainder, as you would in the classical case, and then letting $R= R' \times I$ on this set of measure $\epsilon$.
\end{proof}

\section{Conjugacy Lemma} \label{Sec:ConjugacyLemma}

We now prove a generalization of Halmos' Conjugacy Lemma (see \cite[p.77]{HalmosLectures}), using the same techniques as Halmos' original proof.

\begin{lem}[Conjugacy Lemma for Extensions] \label{Lem:ConjugacyLemmaExtensions}
	Let $T \in \mathcal{G}_X, T_0 \in \mathcal{G}_X$ such that $T'_0$ is antiperiodic, and let $N_{\epsilon}(T)= \{V \in \mathcal{G}_X: m(VD_i \triangle TD_i) < \epsilon, i=1,\ldots, N\}$ be a dyadic neighborhood of $T$. Then there exists $S \in \mathcal{G}_X$ such that $S^{-1}T_0S \in N_{\epsilon}(T).$ 
\end{lem}

\begin{proof}
	Let $k_0 \in \mathbb{N}$ be greater than the ranks of all $D_i$ and $\frac{1}{2^{k_0-2}} < \epsilon.$ Further, let $Q \in N_{\epsilon/2}(T)$, a dyadic permutations of rank $k \ge k_0$ with all properties guaranteed by the Weak Approximation Theorem for Extensions, Theorem \ref{Thm:WATE} ($Q'$ is cyclic, $Q$ is $2^k$ periodic). Applying the Uniform Approximation Theorem for Extensions, Theorem \ref{Thm:UATE}, with $2^k$ in place of $n$ and $\frac{1}{2^k}$ in place of $\epsilon$, there exists $R \in \mathcal{G}_X$ such that $R,R'$ are have period $2^k$ almost everywhere, and $d'(R,T_0) \le \frac{1}{2^k} + \frac{1}{2^k} < \frac{\epsilon}{2}$. 
	
	We will show $Q$ and $R$ are conjugate by some $S \in \mathcal{G}_X.$ Let $q = 2^k$ and  $E_0,\ldots, E_{q-1}$ be cylinder sets of dyadic intervals of rank $k$ in $X,$ arranged so that $QE_i = E_{i+1} ~(i \mod q).$ Note that $m(E_i)=\frac{1}{q}$ By Lemma \ref{Lem:PeriodicExtensionPartition}, there exists $F_0$, a cylinder set, such that $m(F_0)=\frac{1}{q}$ and $F_0,RF_0,\ldots,R^{q-1}F_0$ partition $Z$. Let $F_i:=R^iF_0$. Let $S$ be any measure preserving transformation which maps $E_0$ to $F_0$ as an extension of some $S'.$ Then for $z \in E_i,$ let $Sz:=R^iSQ^{-i}z.$ This can be seen in the following diagram:
	
	\[
	\xymatrix{
		E_0\ar[r]^{Q}\ar[d]_{S}&E_1\ar[r]^{Q}\ar[d]_{S}&E_2\ar[r]^{Q}\ar[d]_{S}&\ldots\ar[r]^{Q}&E_{q-2}\ar[r]^{Q}\ar[d]_{S}&E_{q-1}\ar[d]_{S} \\
		F_0\ar[r]_{R}&F_1\ar[r]_{R}&F_2\ar[r]_{R}&\ldots\ar[r]_{R}&F_{q-2}\ar[r]_{R}&F_{q-1}
	}
	\]
	
	Commutation of the diagram shows that $Q = S^{-1}RS.$ Further, because $Q,R,S\restriction_{E_0}$ are extensions of $Q',R',S'\restriction_{\pi E_0}, S$ is an extension of $S'$.
	
	Now, because $d'$ is invariant under group operations, we have 
	
	$$d'(Q,S^{-1}T_0S) \le d'(S^{-1}RS,S^{-1}T_0S) = d'(R,T_0) < \frac{\epsilon}{2}.$$
	
	 Thus, for any $D_i$ we have:
	
	\begin{align*} \label{Eqn:ConjugacyFinal}
	m(TD_i \triangle S^{-1}T_0SD_i) &\le m(TD_i \triangle QD_i) + m(QD_i \triangle S^{-1}T_0SD_i) \\ 
	&\le \frac{\epsilon}{2} + d(Q,S^{-1}T_0S). 
	\end{align*}
	
	But $d \le d',$ so $m(TD_i \triangle S^{-1}T_0SD_i) \le \frac{\epsilon}{2} + d'(Q,S^{-1}T_0S) < \frac{\epsilon}{2} + \frac{\epsilon}{2} = \epsilon.$
\end{proof}

This lemma is so important because as we will see in our main result in Theorem $\ref{Thm:CategoryTheorem},$ the conjugacy class of $\mathcal{W}_X$ is $\mathcal{W}_X$ itself. Thus, in proving Lemma \ref{Lem:ConjugacyLemmaExtensions}, we have indeed proven half of Theorem \ref{Thm:CategoryTheorem}.

\section{Category Theorem} \label{Sec:CategoryTheorem}

We are fast approaching our main goal: that $\mathcal{W}_X$ is a dense, $G_{\delta}$ subset of $\mathcal{G}_X$. Before we can prove it, we need to prove a few technical results. First we have a quick consequence of the Cauchy-Schwarz Inequality, Proposition \ref{Thm:C-S}, but it will be important enough to make a special note of it.

\begin{prop} \label{Prop:C-SNorms}
	Let $f,g \in L^2(Z|X).$ Then
	
	$$ \norm{\bbe(f g |X)}_{L^2(X)} \le \norm{\norm{f}_{L^2(Z|X)}}_{L^{\infty}(X)} \norm{g}_{L^2(Z)}.$$
\end{prop}

\begin{proof}
	By the Cauchy-Schwarz Inequality for $L^2(Z|X)$, we have

	$$\abs{\bbe(f g |X)} \le \bbe(\abs{f}^2|X)^{1/2} \bbe(\abs{g}^2|X)^{1/2}$$

	pointwise. Now, by definition of $L^2(Z|X), \bbe(\abs{f}^2|X)^{1/2} \in L^{\infty}(X).$ Letting $M := \norm{\norm{f}_{L^2(Z|X)}}_{L^{\infty}(X)},$ we have 
	
	\begin{equation} \label{Eq:C-S}
	\abs{\bbe(f g |X)} \le M  \bbe(\abs{g}^2|X)^{1/2}.
	\end{equation}
	
	Now note that by Fubini,
	
	\begin{align*}
	 \norm{\bbe(\abs{g}^2|X)^{1/2}}_{L^2(X)}
	&= \ \left(\int_X \left(\left( \int_Y \abs{g}^2(x,y) d \mu_Y \right)^{1/2}\right)^2 d \mu_X \right)^{1/2} \\ &= \ \left( \int_Z \abs{g}^2 dm_2 \right)^{1/2} = \norm{g}_{L^2(Z)},
	\end{align*}
	
	where $Y=X, \mu_X=\mu_Y = m_1$ (the notation here was changed to clarify what integrals were intended). And so taking $L^2(X)$ norm on both sides of (\ref{Eq:C-S}), we arrive at the desired inequality.
\end{proof}

Our next goal is to prove that $T \in \mathcal{W}_X$ is equivalent to the existence of a subsequence $n_k$ such that for all $f,g \in L^2(Z|X),$

\begin{equation} \label{Eqn:RelWMSubsequence}
\lim_{k \to \infty} \norm{\bbe(T^{n_k}f \overline{g} |X) - (T')^{n_k} \bbe(f|X)\bbe(\overline{g}|X)}_{L^2(X)} = 0. 
\end{equation}

To this end, we first show that if (\ref{Eqn:RelWMSubsequence}) holds for an $L^2(Z)$-dense subset of $L^2(Z|X),$ it holds for all of $L^2(Z|X).$

\begin{lem} \label{Lem:RWML2Dense}
	Let $T \in \mathcal{G}_X,$ and let $D \subset L^2(Z|X),$ with $\norm{\norm{\overline{f}}_{L^2(Z|X)}}_{L^{\infty}(X)} \le 1$ for all $f \in D,$ such that $D$ is dense in the unit ball of $L^2(Z|X),$ but with respect to the $L^2(Z)$ norm topology. Further suppose that there exists a subsequence $n_k$ such that for all $f_i,f_j \in D,$
	
	$$\lim_{k \to \infty} \norm{\bbe(T^{n_k}f_i \overline{f_j}|X) - (T')^{n_k}\bbe(f_i|X)\bbe(\overline{f_j}|X)}_{L^2(X)} = 0.$$
	
	Then for all $h,g \in L^2(Z|X)$
	
	$$\lim_{k \to \infty} \norm{\bbe(T^{n_k}h \overline{g}|X) - (T')^{n_k}\bbe(h|X)\bbe(\overline{g}|X)}_{L^2(X)} = 0.$$
\end{lem}

\begin{proof}
	Fix $h,g \in L^2(Z|X)$ and let $\{h_j\}, \{g_j\} \subset D$ such that 
	
	$$h_j \xrightarrow{L^2(Z)} h, g_j \xrightarrow{L^2(Z)} g.$$
	
	We first claim that
	
	$$\lim_{j \to \infty} \norm{\bbe(T^nh_j \overline{g_j}|X) - \bbe(T^nh \overline{g}|X)}_{L^2(X)} = 0,$$
	
	uniformly with respect to $n$. Indeed, we have 
	
	\begin{align*}
	&\quad\ \ \norm{\bbe(T^nh_j \overline{g_j}|X) - \bbe(T^nh \overline{g}|X)}_{L^2(X)} \\
	&\le \ \norm{\bbe(T^nh_j \overline{g_j}- T^n h_j \overline{g}|X)}_{L^2(X)} + \norm{\bbe(T^n h_j \overline{g} - T^nh \overline{g}|X)}_{L^2(X)} \\
	&= \ \norm{\bbe(T^nh_j (\overline{g_j}- \overline{g})|X)}_{L^2(X)} + \norm{\bbe(\overline{g}(T^n h_j  - T^nh)|X)}_{L^2(X)}.
	\end{align*}
	
	Now by Proposition \ref{Prop:C-SNorms}
	
	\begin{align*}
	\norm{\bbe(T^nh_j (\overline{g_j}- \overline{g})|X)}_{L^2(X)} &\le \norm{\norm{T^nh_j}_{L^2(Z|X)}}_{L^{\infty}(X)} \norm{\overline{g_j}- \overline{g}}_{L^2(Z)}, \\
	\norm{\bbe(\overline{g}(T^n h_j  - T^nh)|X)}_{L^2(X)} &\le \norm{\norm{\overline{g}}_{L^2(Z|X)}}_{L^{\infty}(X)} \norm{T^n(h-h_j)}_{L^2(Z)}.
	\end{align*}
	
	In turn, as $\norm{\norm{T^nh_j}_{L^2(Z|X)}}_{L^{\infty}(X)} = \norm{\norm{h_j}_{L^2(Z|X)}}_{L^{\infty}(X)}$ and $T$ is an isometry, we get
	
		\begin{align*}
	\norm{\norm{T^nh_j}_{L^2(Z|X)}}_{L^{\infty}(X)} \norm{\overline{g_j}- \overline{g}}_{L^2(Z)} &\le \norm{\norm{h_j}_{L^2(Z|X)}}_{L^{\infty}(X)} \norm{\overline{g_j}- \overline{g}}_{L^2(Z)}, \\
	\norm{\norm{\overline{g}}_{L^2(Z|X)}}_{L^{\infty}(X)} \norm{T^n(h-h_j)}_{L^2(Z)} &\le \norm{\norm{\overline{g}}_{L^2(Z|X)}}_{L^{\infty}(X)} \norm{h-h_j}_{L^2(Z)}.
	\end{align*}

	Because $h_j \to h, g_j \to g$ in $L^2(Z),$ we have the desired result. Note that a similar argument using Proposition \ref{Prop:C-SNorms} will show $\bbe(\overline{g_j}|X) \to \bbe(\overline{g}|X), (T')^n\bbe( h_j |X) \to (T')^n\bbe(h|X))$ in $L^2(X).$
	
	Now,
	
	\begin{align*}
	&\quad\ \norm{\bbe(T^nh \overline{g}|X) - (T')^n\bbe(h|X)\bbe(\overline{g}|X)}_{L^2(X)} \\
	&\le \ \norm{\bbe(T^nh \overline{g}|X) - \bbe(T^nh_j \overline{g_j}|X)}_{L^2(X)} \\
	&+ \ \norm{\bbe(T^nh_j \overline{g_j}|X) - (T')^n\bbe(h_j|X)\bbe(\overline{g_j}|X)}_{L^2(X)} \\
	&+ \ \norm{(T')^n\bbe(h_j|X)\bbe(\overline{g_j}|X) - (T')^n\bbe(h|X)\bbe(\overline{g}|X)}_{L^2(X)}.
	\end{align*}
	
	By hypothesis, there is a subsequence $n_k$ (independent of $j$) such that the middle term converges to 0. Further, the first and third terms converge to 0 as $j \to \infty$ uniformly in $n$, so 
	
	$$\lim_{k \to \infty} \norm{\bbe(T^{n_k}h \overline{g}|X) - (T')^{n_k}\bbe(h|X)\bbe(\overline{g}|X)}_{L^2(X)} = 0$$
	
	as desired.
\end{proof}

Next, recall that a function $f \in L^2(Z|X)$ is called \textit{generalized eigenfunction} for a given $T \in \mathcal{G}_X$ (see \cite[p.179]{Glasner}) if the $L^{\infty}(X)$-module spanned by $\{T^nf: n \in \mathbb{N}\}$ has finite rank. In other words, there exists $g_1, \ldots, g_l \in L^2(Z|X)$ such that for all $n$, there exists $c^n_j \in L^{\infty}(X), 1 \le j \le l$ such that

$$T^nf(x,y) = \sum_j c^n_j(x) g_j(x,y).$$

\begin{lem} \label{Lem:RelWMSubsequence}
Let $T \in \mathcal{G}_X$. Then $T \in \mathcal{W}_X$ if and only if there exists a subsequence $n_k$ such that for all $f,g \in L^2(Z|X)$

\begin{equation} \label{Eqn:RelWMSubsequence2}
\lim_{k \to \infty} \norm{\bbe(T^{n_k}f \overline{g} |X) - (T')^{n_k} \bbe(f|X)\bbe(\overline{g}|X)}_{L^2(X)} = 0. 
\end{equation}
\end{lem}

\begin{proof}
Let $T \in \mathcal{W}_X.$ Lemma \ref{Lem:RWML2Dense} tells us that we need only show that there exists a subsequence such that (\ref{Eqn:RelWMSubsequence2}) holds for an $L^2(Z)$-dense subset of $L^2(Z|X)$. Let $(f_i)_{i=1}^{\infty}$ be an enumeration of that dense subset, and let $l \mapsto (f_i,f_j), 1 \le l < \infty$ order the set $\{ (f_i,f_j) \}$ arbitrarily. Now, by the definition of a weakly mixing extension and the Koopman-von Neumann Lemma (see, e.g., \cite[p.54]{EinsiedlerWard}), for each $l$ there exists a subsequence $n'^l_{m}$ of upper asymptotic density 1 such that 

$$\lim_{m \to \infty} \norm{\bbe(T^{n'^l_{m}}f_i \overline{f_j} |X) - (T')^{n'^l_{m}} \bbe(f_i|X)\bbe(\overline{f_j}|X)}_{L^2(X)} = 0.$$

Stated differently, for the pair $f_i,f_j$ corresponding to $l$, we have the desired convergence along the dense subsequence $n'^l_{m}$. We now define a new density 1 subsequence for each $l$ inductively. Let $n^1_{m}:= n'^1_m,$ and for $l > 1,$ define $n^l_m$ to be a common density 1 subsequence of $n^{l-1}_m$ and $n'^l_m.$ Now define $n_k$ be a diagonal sequence obtained from $n^l_m,$ i.e., $n_k := n^k_k.$ It is easy to see that for all $f_i,f_j,$ (\ref{Eqn:RelWMSubsequence2}) holds.

To prove the converse, suppose $T \in \mathcal{G}_X \backslash \mathcal{W}_X$. Then there exists $f \in L^2(Z|X) \backslash L^{\infty}(X)$ that is a generalized eigenfunction for $T$, see \cite[p.192]{Glasner}. Without loss of generality, $\norm{f}_{L^1(Z)}=1.$ Let $g_1, \ldots, g_l \in L^2(Z|X)$ be a basis for the module spanned by $T^nf.$ We want $g_i$ to be ``relatively orthonormal''. That is, we want $\bbe(g_i \overline{g}_j|X)=0 \ a.e.$ when $i \neq j$ and $\bbe(\abs{g_j}^2|X)=1 \ a.e.$ This can be accomplished with a relative Gram-Schmidt process. We start by defining $h'_1:= g_1.$ For any $x$ such that $\bbe(\abs{h'_1}^2|X)(x)=0,$ we have that $h'_1(x,y):=h'_{1,x}(y) \equiv 0.$ Thus by setting the corresponding $c^n_1(x)=0$ for all $n$ we can define $h_{1,x}(y)$ arbitrarily, so long as it's not identically 0. For all other $x$, define $h_{1,x}(y):= h'_{1,x}(y)$. Now, having defined $h_j, 1 \le j \le i-1,$ we define $h'_i$ by

$$h'_i := g_i - \sum_{j=1}^{i-1} \frac{\bbe(g_i \overline{h}_j|X)}{\bbe(\abs{h_j}^2|X)}h_j.$$

Similar to the above, if there are any $x$ such that $\bbe(\abs{h'_i}^2|X)(x)=0,$ we define $h_{i,x}(y) \not\equiv 0$ (again changing $c^n_i(x)$ to 0 for all $n$) and for all other $x, h_{i,x}(y) := h'_{i,x}(y)$. Finally we normalize and redefine $g_j$ so that

$$g_j(x,y):= \frac{h_j(x,y)}{\bbe(\abs{h_j}^2|X)(x)^{1/2}}.$$

Now define a function $j : \mathbb{N} \to \{1, \ldots, l\}$ such that $\norm{c^n_{j(n)}}_{L^2(Z)} \ge \norm{c^n_{i}}_{L^2(Z)}, 1 \le i \le l.$ Note that for each $n, \norm{c^n_{j(n)}}_{L^2(Z)} \ge 1/l$ as else

$$\norm{T^nf}_{L^1(Z)} \le \sum_i \norm{c^n_i g_i}_1 \le \sum_i \norm{c^n_i}_2 \norm{g_i}_2 < 1.$$

Fix $n$ and suppose for now that $\bbe(f|X) \equiv 0$ almost everywhere. Note that this is guaranteed to be possible because if $f= f_0 + \bbe(f|X)$ is a generalized eigenfunction with basis $\{g_1, \ldots g_l\},$ then $f_0$ is a generalized eigenfunction with spanning set $\{g_1, \ldots, g_l, \bbe(f|X)\},$ and $\bbe(f_0|X) \equiv 0$ by design. Now, by relative orthonormality we have

\begin{align*}
&\quad \ \norm{\bbe(T^nf \overline{g}_{j(n)} |X) - (T')^n \bbe(f|X)\bbe(\overline{g}_{j(n)}|X)}_{L^2(X)} \\
&= \norm{\bbe(T^nf \overline{g}_{j(n)} |X)}_{L^2(X)} \\
&= \norm{\bbe\left(\left(\sum_{i=1}^l c^n_i g_i\right) \overline{g}_{j(n)} |X \right)}_{L^2(X)} \\
&= \norm{\bbe(c^n_{j(n)} g_{j(n)} \overline{g}_{j(n)} |X)}_{L^2(X)} \\
&= \norm{c^n_j \bbe(\abs{g_{j(n)}}^2|X)}_{L^2(X)} \\
&= \norm{c^n_j}_{L^2(Z)} \ge \frac{1}{l}.
\end{align*}

Define $B_i := j^{-1}(i)$. By the above work, if $n \in B_i,$

$$\norm{\bbe(T^nf \overline{g}_i |X)}_{L^2(X)} \ge \frac{1}{l}.$$

Now given any subsequence $(n_k)$ there will be at least one $i \in \{\1, \ldots, l\}$ such that $(n_k)$ intersects $B_i$ infinitely often. Thus, $\norm{\bbe(T^{n_k}f \overline{g_i} |X)}$ does not converge to 0 as $k \to \infty$.

If $\bbe(f|X) \neq 0,$ we write $f = f_0 + h$ where $\bbe(f_0|X) = 0$ and $h := \bbe(f|X)$. Then 

\begin{align*}
&\bbe(T^nf \overline{g} |X) - (T')^n \bbe(f|X)\bbe(\overline{g}|X) \\
= \ &\bbe(T^n(f_0 + h) \overline{g} |X) - (T')^n \bbe(f_0 + h|X)\bbe(\overline{g}|X).
\end{align*}

By linearity of the conditional expectation, this is the same as

$$\bbe(T^nf_0 \overline{g} |X) - (T')^n \bbe(f_0|X) \bbe(\overline{g}|X) +  \bbe(T^n h \overline{g}|X) - (T')^n \bbe(h|X) \bbe(\overline{g}|X).$$

The second term is 0 as $\bbe(f_0|X) \equiv 0.$ Further, because $h \in L^{\infty}(X), \bbe(T^n h \overline{g}|X) = (T')^n h \bbe(\overline{g}|X) = (T')^n \bbe(h|X) \bbe(\overline{g}|X),$ which cancels with the fourth term. We are left with $\bbe(T^nf_0 \overline{g} |X)$ and have reduced this to the previous case.
\end{proof}

Finally, we arrive at our goal.

\begin{thm}[Weakly Mixing Extensions are Residual] \label{Thm:CategoryTheorem}
$\mathcal{W}_X$ is a dense, $G_{\delta}$ subset of $\mathcal{G}_X$. 
\end{thm}

\begin{proof}
We begin by proving that if $T \in \mathcal{W}_X$ and $S \in \mathcal{G}_X,$ then $S^{-1}TS \in \mathcal{W}_X.$ With the note that there are weakly mixing extensions of antiperiodic factors, we then use Lemma \ref{Lem:ConjugacyLemmaExtensions} to conclude $\mathcal{W}_X$ is dense. 

We need to prove that there exists a subsequence of

\begin{align*}
& \norm{\bbe((S^{-1}TS)^nf \overline{g} |X) - ((S')^{-1}T'S')^n \bbe(f|X)\bbe(\overline{g}|X)}_{L^2(X)} \\
= \ &\norm{\bbe(S^{-1}T^nSf \overline{g} |X) - (S')^{-1}(T')^nS' \bbe(f|X)\bbe(\overline{g}|X)}_{L^2(X)}
\end{align*}

which converges to 0. Indeed, the above equals

\begin{align*}
\ &\norm{\bbe(S^{-1}T^nSf (S^{-1} S)\overline{g} |X) - (S')^{-1}(T')^nS' \bbe(f|X)((S')^{-1} S')\bbe(\overline{g}|X)}_{L^2(X)} \\
= \ &\norm{(S')^{-1}\bbe(T^n(Sf) (S\overline{g}) |X) - (S')^{-1}(T')^n\bbe(Sf|X) \bbe(S\overline{g}|X)} \\
\le \ &\norm{(S')^{-1}} \norm{\bbe(T^n(Sf) (S\overline{g}) |X) - (T')^n\bbe(Sf|X) \bbe(S\overline{g}|X)} \\
= \ &\norm{\bbe(T^n(Sf) (S\overline{g}) |X) - (T')^n\bbe(Sf|X) \bbe(S\overline{g}|X)}.
\end{align*}

But $T$ is a weakly mixing extension of $T'$ so there exists a subsequence for which the above converges to 0. 

To prove that $\mathcal{W}_X$ is $G_{\delta},$ let $\{f_i\} \subset L^2(Z|X)$ be dense with respect to $L^2(Z)$ and for $i,j,k,n \in \mathbb{N},$ consider the sets

$$A_{i,j,k,n} := \left\{S \in \mathcal{G}_X : \norm{\bbe(S^nf_i \overline{f_j} |X) - (S')^n \bbe(f_i|X)\bbe(\overline{f_j}|X)}_{L^2(X)} < \frac{1}{k} \right\}.$$

Due to Lemmas \ref{Lem:RWML2Dense}, \ref{Lem:RelWMSubsequence}, we see that $\bigcap_{i,j,k} \bigcup_{n \ge k} A_{i,j,k,n} = \mathcal{W}_X.$

Thus it is sufficient to prove that each $A_{i,j,k,n}$ is open. For this, we show that for fixed $n \in \mathbb{N},f,g \in L^2(Z|X)$ and $\epsilon > 0,$ the set

$$ \{S \in \mathcal{G}_X : \norm{\bbe(S^nf \overline{g} |X) - (S')^n\bbe(f|X)\bbe(\overline{g}|X)} < \epsilon \}$$

is open in the weak topology. To this end, we show that the complement

$$V(n,f,g,\epsilon) := \{S \in \mathcal{G}_X : \norm{\bbe(S^nf \overline{g} |X) - (S')^n\bbe(f|X)\bbe(\overline{g}|X)} \ge \epsilon \}$$

is closed. Let $(S_m) \subset V(n,f,g,\epsilon)$ be a sequence of Koopman operators with $(S_m)$ converging weakly to a Koopman operator $S$. Note that this implies that $S_m \to S$ strongly (as Koopman operators are all isometries). 

First note that in general, if we have functions $g,h,h_1, h_2, \ldots \in L^2(Z|X),$ and $h_m \to h$ in $L^2(Z),$ then $\bbe(h_mg|X) \to \bbe(hg|X)$ in $L^2(X).$ Indeed, by Proposition \ref{Prop:C-SNorms},

\begin{align*}
\norm{\bbe(gh_m|X) - \bbe(gh|X)}_{L^2(X)} &= \norm{\bbe(g(h-h_m)|X)}_{L^2(X)} \\
&\le \norm{\norm{g}_{L^2(Z|X)}}_{L^{\infty}(X)} \norm{h_m-h}_{L^2(Z)}.
\end{align*}

Second, note that $S_m \to S$ strongly implies $(S_m)' \to S'$ strongly. To see this, let $h \in L^2(X)$ and let $\hat{h} \in L^2(Z)$ be defined so that $\hat{h}(x,y):=h(x)$ (that is, $\hat{h}$ is constant on fibers). Then by Fubini, $\norm{(S_m)'h-S'h}_{L^2(X)} = \norm{S_m \hat{h} - S \hat{h}}_{L^2(Z)}.$

Lastly note that if $S_m \to S$ strongly, then $S_m^n \to S^n$ strongly.

With these facts, we see that $\bbe(S_m^nf \overline{g} |X) - (S_m')^n\bbe(f|X)\bbe(\overline{g}|X)$ converges to $\bbe(S^nf \overline{g} |X) - (S')^n\bbe(f|X)\bbe(\overline{g}|X)$ strongly. Thus, $S \in V(n,f,g,\epsilon),$ and so $V(n,f,g,\epsilon)$ is closed.
\end{proof}

\begin{remark}
Our assumption that $(X,m)$ was the Lebesgue measure on the unit interval was only important for the proof of density, where we need a non-atomic probability space to have an antiperiodic factor. Proving $\mathcal{W}_X$ is $G_{\delta}$ never required anything of $X$, and so the proof will hold for $(X,m)$ being replaced by any probability space.
\end{remark}

\section{General Case} \label{Sec:General}
There is still a case left to consider. Namely, the case where the ``vertical'' measure is neither purely non-atomic nor purely discrete. Let $(X,m,T')$ be as before with $T'$ invertible, and let $(Y, \eta)$ be a probability space where $Y = A \dot{\cup} B$, where $B$ is an at most countable set, each of which is an atom of $\eta$, and $\eta \restriction_{A}$ is non-atomic (in particular, $0 < \eta(A) < 1$). Let $(Z,\mu,T)=(X \times Y, m \times \eta,T),$ where $T \in \mathcal{G}_X$ is an extension of $T'.$ Let $C:=X \times A$ and $D:=X \times B$

We first show that $T$ cannot mix points on the discrete and non-discrete parts of $Y$.

\begin{prop} \label{Prop:GeneralNoPointMixing}
Let $C,D \subset Z$ be as defined above. Then up to a set of measure zero, $TC \subset C$ and $TD \subset D.$
\end{prop}

\begin{proof}
	First, suppose there exists $D' \subset D$ such that $\mu(D') > 0$ and $TD' \subset C$. We can assume without loss of generality that there exists $k$, a level of $D,$ such that $D' \subset k$ (see Section \ref{Sec:Discrete} for an explanation of this notation). We claim that $m(TD')=0,$ so that $T$ is not measure preserving. Note that because $T$ is an extension of invertible $T'$ and $D'$ is contained to a single level of $D$, then for each fixed $x \in X, TD' \cap \pi^{-1}(x)$ contains at most one point. Therefore, by Fubini
	
	$$\mu(TD') = \int_Z \chi_{TD'} d\mu = \int_X \left( \int_Y \chi_{TD'}(x,y) d\eta \right) dm.$$
	
	But by our previous note and the fact that $(A, \eta \restriction_A)$ is non-atomic, $\int_Y \chi_{TD'}(x,y) d\eta =0$ for all $x$, and thus $\mu(TD')=0.$
	
	Now suppose there exists $C' \subset C$ such that $\mu(C')>0$ and $TC' \subset D.$ Note that there exists $x_0 \in X$ such that $\pi^{-1}(x_0) \cap C'$ is uncountable (else $\mu(C')=0$ with a similar argument as above). But $T(\pi^{-1}(x_0) \cap C') \subset T'x_0 \times B$ is a countable set. This contradicts the invertibility of $T$.
\end{proof}

A quick consequence of Proposition \ref{Prop:GeneralNoPointMixing} is that there are no weakly mixing extensions on $Z$.

\begin{cor} \label{Cor:NoWeakMixingOnMixedVertical}
Let $(X,m,T'), (Y,\eta), (Z,\mu,T), C,$ and $D$ be as defined above, with $A$ and $B$, continuous and discrete parts of $Y$, respectively, nonempty. Then $T \notin \mathcal{W}_X$.
\end{cor}

\begin{proof}
Note that we can assume without loss of generality that $TC \subset C$ and $TD \subset D$ (that is, not up to a set of measure 0, but rather everywhere). Define $f(z)$ as

\begin{center}
$ f(z) := \left\{
\begin{array}{cc}
\frac{1}{\eta(A)} & \text{if } z \in C \\
\frac{-1}{\eta(B)} & \text{if } z \in D
\end{array} \right.
$.
\end{center}

Clearly $f \in L^2(Z|X)$ and simple calculation shows that $f$ has relative mean zero. Further, by Proposition \ref{Prop:GeneralNoPointMixing}, $f$ is $T$-invariant, and so 

$$\bbe(T^nf \overline{f}|X)=\bbe(f \overline{f}|X) = \bbe(\abs{f}^2|X) > 0.$$

Therefore, $T$ is not a weakly mixing extension of $T'$.
\end{proof}

\begin{remark}
As with the $G_{\delta}$ part of Theorem \ref{Thm:CategoryTheorem}, Proposition \ref{Prop:GeneralNoPointMixing} and Corollary \ref{Cor:NoWeakMixingOnMixedVertical} hold when $(X,m)$ is replaced with any standard probability space.
\end{remark}

\section{Strongly Mixing Extensions} \label{Sec:StrongMixing}
In this section we first extend the notion of strongly mixing transformations to extensions, just as the notions of ergodic and weakly mixing transformations were extended to extensions. Afterwards we will show that the set of strongly mixing extensions form a set of first category in $\mathcal{G}_X$.

\begin{defin} \label{Def:StronglyMixing}
Let $(X,\nu), (Z,\mu)$ be  probability spaces. We say that $T \in \mathcal{G}_X$ is a \textit{(strongly) mixing extension} of $T'$ or $T$ is \textit{(strongly) mixing relative to} $T'$ if for all $f,g \in L^2(Z|X),$

$$\lim_{n \to \infty} \norm{\bbe(T^nf \overline{g} |X) - (T')^n \bbe(f|X)\bbe(\overline{g}|X)}_{L^2(X)} = 0.$$

Let $\mathcal{S}_X \subset \mathcal{G}_X$ denote the set of strongly mixing extensions.
\end{defin}

Definition \ref{Def:StronglyMixing} yields some of the properties one would hope to have from the extension of the notion of strongly mixing transformations. For example, $\mathcal{S}_X$ is in general not empty. Indeed, any direct product transformation where the second component is strongly mixing will be a strongly mixing extension. We also have that if $X$ is a single point, then the definition coincides with classical strongly mixing transformations. Further, it is clear that $\mathcal{S}_X \subset \mathcal{W}_X$.

We once again return to the case where $(X,m)$ is the unit interval with the Lebesgue measure, and $(Z,m)$ is the unit square with the Lebesgue measure. Analogous to Rokhlin's result and its proof, we now show that $\mathcal{S}_X$ is a first category subset of $\mathcal{G}_X$.

\begin{thm}[Strongly Mixing Extensions are of First Category] \label{Thm:FirstCategoryTheorem}
$\mathcal{S}_X \subset \mathcal{G}_X$ is of first category.
\end{thm}

\begin{proof}
For $k \in \mathbb{N},$ let $P_k := \{T \in \mathcal{G}_X | T^k = I_Z \}.$ Note in particular that if $T \in P_k$ then $(T')^k = I_X.$ For $n \in \mathbb{N},$ let $\hat{P}_n := \bigcup_{k > n} P_k.$ Note that the Weak Approximation Theorem for Extensions (Theorem \ref{Thm:WATE}) implies that $\hat{P}_n$ is dense in $\mathcal{G}_X$.

Let $A:= [0,1] \times [0, 1/2]$ (the bottom half of $Z$). Note that $\bbe(\chi_A |X) = 1/2$ for all $x \in X$. We now define new sets,

$$M_k := \left\{T \in \mathcal{G}_X | \norm{\bbe(T^k \chi_A \chi_A|X) - (T')^k \bbe(\chi_A|X) \bbe(\chi_A|X)}_{L^2(X)} \le \frac{1}{5}  \right\}.$$

Using the same arguments as used in the proof of Theorem \ref{Thm:CategoryTheorem} for the sets $V(n,f,g,\epsilon),$ we see that $M_k$ is closed for all $k$. Now let 

$$M := \bigcup_{n=1}^{\infty} \bigcap_{k > n} M_k.$$

It is easy to see that $\mathcal{S}_X \subset M$. Thus it is sufficient to show that $M$ is of first category. It is in turn sufficient to show that $\bigcap_{k > n} M_k$ is nowhere dense for all $n$, and further given that $\bigcap_{k > n} M_k$ is closed for all $n$, it is thus sufficient to show that $\mathcal{G}_X \backslash \bigcap_{k > n} M_k$ is dense. Lastly, as $\mathcal{G}_X \backslash \bigcap_{k > n} M_k = \bigcup_{k > n} (\mathcal{G}_X \backslash M_k),$ it will suffice to show that $P_k \subset (\mathcal{G}_X \backslash M_k)$ for all $k$ as then $\hat{P}_n = \bigcup_{k > n} P_k \subset \bigcup_{k > n} (\mathcal{G}_X \backslash M_k)$ and $\hat{P}_n$ is dense.

Now, if $T \in P_k,$ then $T^k = I_Z, (T')^k = I_X,$ so

\begin{align*}
&\quad\ \norm{\bbe(T^k \chi_A \chi_A|X) - (T')^k \bbe(\chi_A|X) \bbe(\chi_A|X)}_{L^2(X)} \\
&= \norm{\bbe(\chi_A|X) - \bbe(\chi_A|X)^2} = \norm{\frac{1}{2}-\frac{1}{4}} = \frac{1}{4} > \frac{1}{5}.
\end{align*}

Thus, $T \notin M_k$.
\end{proof}

\begin{cor} \label{Cor:StrongAndWeakMixingExtensionsNotEqual}
$\mathcal{S}_X$ is a proper subset of $\mathcal{W}_X.$
\end{cor}

\section{Further Questions} \label{Sec:Questions}
To conclude this paper, we formulate some open questions.

\begin{question} \label{Que:AtomicFactor}
	Let $(X,\nu)$ be any probability space, $(Y,\eta)$ be a non-atomic probability space and $(Z,\mu) = (X \times Y, \nu \times \eta).$ Is $\mathcal{W}_X$ a dense, $G_{\delta}$ subset of $\mathcal{G}_X$?
\end{question}

It would be sufficient for Question \ref{Que:AtomicFactor} to consider the case where $(X,\nu)$ is purely atomic. We cannot use the Conjugacy Lemma in this case because there are no antiperiodic transformations on a discrete set.

\begin{question} \label{Que:FixedFactor}
Let $(X,m)$ be a (potentially non-atomic) probability space,  $(Y,\eta)$ be a non-atomic probability space and $(Z,\mu) = (X \times Y, \nu \times \eta).$ Let $R \in \mathcal{G}(X)$ be fixed. Define 
\vspace{-5pt}
$$\mathcal{G}_R := \{T \in \mathcal{G}_X : T \text{ extends } R \}$$ 

and 
\vspace{-5pt}
$$\mathcal{W}_R := \{T \in \mathcal{G}_R : T \text{ is a weakly mixing extension of } R \}.$$ Is $\mathcal{W}_R$ a dense, $G_{\delta}$ subset of $\mathcal{G}_R$? 

Similarly let 
\vspace{-5pt}
$$\mathcal{S}_R := \{T \in \mathcal{G}_R : T \text{ is a strongly mixing extension of } R \}.$$

Is $\mathcal{S}_R$ a first category subset of $\mathcal{G}_R?$
\end{question}

The main difficulty as of now in answering Question \ref{Que:FixedFactor} for weakly mixing extensions is proving the density. The freedom of having a non-fixed factor allowed a generalization of Halmos' Conjugacy Lemma. One cannot conjugate in $\mathcal{G}_R$ unless $R$ is the identity on $X$, as $\mathcal{G}_R$ is not closed under inverses for any other invertible $R$. That is not even to mention lack of dyadic permutations or other tools which we had at our disposal throughout this paper. Even strongly mixing extensions at first present some trouble, as the proof relied on period $T$, and if $R$ is not periodic, then neither is $T$.

We note that in the case of compact group extensions of a fixed weakly mixing $R$, Robinson gave a positive answer to the question of genericity of weakly mixing extensions \cite{EARobertson}.

Further, we see from the following proposition that these Question \ref{Que:FixedFactor} cannot be derived from Theorems \ref{Thm:CategoryTheorem} and \ref{Thm:FirstCategoryTheorem}. Recall that $(X,m)$ is the unit interval with the Lebesgue measure.

\begin{prop}
	For all $T' \in \mathcal{G}(X), \mathcal{G}_{T'}$ is a closed, nowhere dense subset of $\mathcal{G}_X.$
\end{prop}

\begin{proof}
	Fix $T' \in \mathcal{G}(X).$ We first show that $\mathcal{G}_{T'}$ is closed. Let $S \in \mathcal{G}_X \backslash \mathcal{G}_{T'}.$ As $S' \neq T',$ there exists $E \subset X$ such that $m(S'E \triangle T'E) > 0.$ Let $\epsilon := m(S'E \triangle T'E) /2.$
	
	Consider 
	
	$$N_{\epsilon}(S) := \{R \in \mathcal{G}_X : \mu(R \pi^{-1}E \triangle S \pi^{-1}E) < \epsilon \}.$$
	
	For any $T \in \mathcal{G}_{T'}, \mu(T \pi^{-1}E \triangle S \pi^{-1}E) = m(T'E \triangle S'E) = 2 \epsilon > \epsilon,$ so $T \notin N_{\epsilon}(S),$ and thus $\mathcal{G}_{T'}$ is closed.
	
	Now, to show $\mathcal{G}_{T'}$ is nowhere dense, we show  $\mathcal{G}_X \backslash \mathcal{G}_{T'}$ is dense. Fix $T \in \mathcal{G}_{T'}, \epsilon > 0$ and let
	
	$$N_{\epsilon}(T) := \{S \in \mathcal{G}_X : \mu(TD_i \triangle SD_i) < \epsilon, i = 1, \ldots 2^{2N} \}$$
	
	where $D_i$ are all dyadic squares of some rank $N$. Fix $i$ and let $E \subset \pi D_i$ with $m(E) < \epsilon$. Let $E_1, E_2$ be disjoint sets whose union is $E$ and $m(E_1)=m(E_2).$
	
	Select $R \in \mathcal{G}(X)$ with the following properties: for $x \in X \backslash E, Rx = x,RE_1=E_2, RE_2 = E_1.$ Define $\tilde{T} := T (R \times I)$. Note that for $j$ such that $\pi D_j \neq \pi D_i, (R \times I)D_j = D_j,$ so for such $j, \mu(\tilde{T}D_j \triangle TD_j) = 0.$ On the other hand, if $\pi D_j = \pi D_i,$ then $\mu(\tilde{T}D_j \triangle TD_j) < m(E) = \epsilon.$ Thus $\tilde{T} \in N_{\epsilon}(T),$ but $\tilde{T} \notin \mathcal{G}_{T'}.$
\end{proof}

\begin{question} \label{Que:Rigidity}
Can one find an extension analogue of Katok's result (see, for example, \cite{Nadkarni}) that rigid transformations form a residual set?
\end{question}

Note that is not completely clear how one should define a rigid extension. The author currently has no notion of how it should be defined.

%\begin{question} \label{Que:NoProudct}
%Suppose that $(Z,\mu,S)$ is an extension of $(X,\nu,S')$ through a factor map $\phi: Z \to X$. Let $\mathcal{G}_X, \mathcal{W}_X$ be as defined previously, now with $\phi$ as the fixed factor map (the important distinction now being that we are not considering these extensions as being on product spaces through natural projections). Once again, is $\mathcal{W}_X$ a dense, $G_{\delta}$ subset of $\mathcal{G}_X$?
%\end{question}

\nocite{Glasner}
\nocite{HalmosPaper}
\nocite{Peterson}
\nocite{Zhao}

\bibliographystyle{plain}
\bibliography{CategoryTheorem} 

\begin{thebibliography}{10}

\bibitem{Ageev1}
O.N. Ageev.
\newblock The generic automorphism of a {L}ebesgue space conjugate to a
  {G}-extension for any finite abelian group {G}.
\newblock {\em Doklady Akademii Nauk}, 374:439--442, 2000.

\bibitem{Ageev2}
O.N. Ageev.
\newblock On the multiplicity function of generic group extensions with
  continuous spectrum.
\newblock {\em Ergodic Theory and Dynamical Systems}, 21:321--338, 2001.

\bibitem{Ageev3}
O.N. Ageev.
\newblock On the genericity of some nonasymptotic dynamic properties.
\newblock {\em Uspekhi Matematicheskikh Nauk}, 58:177--178, 2003.

\bibitem{Ageev4}
O.N. Ageev.
\newblock The homogeneous spectrum problem in ergodic theory.
\newblock {\em Inventiones mathematicae}, 160:417--446, 2005.

\bibitem{AlpernPrasad}
Steve Alpern and V.S. Prasad.
\newblock Properties generic for {L}ebesgue space automorphisms are generic for
  measure-preserving manifold homeomorphisms.
\newblock {\em Ergodic Theory and Dynamical Systems}, 22:1587--1620, 2002.

\bibitem{AssaniDuncanMoore}
Idris Assani, David Duncan, and Ryo Moore.
\newblock Pointwise characteristic factors for {W}iener-{W}intner double
  recurrence theorem.
\newblock {\em Ergodic Theory and Dynamical Systems}, 36:1037--1066, 2016.

\bibitem{AssaniPresser}
Idris Assani and Kimberly Presser.
\newblock Pointwise characteristic factors for the multiterm return times
  theorem.
\newblock {\em Ergodic Theory and Dynamical Systems}, 32:341--360, 2012.

\bibitem{BergelsonTaoZiegler}
Vitaly Bergelson, Terence Tao, and Tamar Ziegler.
\newblock Multiple recurrence and convergence results associated to
  ${F}_{P}^{\omega}$-actions.
\newblock {\em Journal d'Analyse Math\'{e}matique}, 127:329--378, 2015.

\bibitem{Chacon}
R.V. Chacon.
\newblock Weakly mixing transformations which are not strongly mixing.
\newblock {\em Proceedings of the American Mathematical Society}, 22:559--562,
  1969.

\bibitem{Chu}
Qing Chu.
\newblock Convergence of weighted polynomial multiple ergodic averages.
\newblock {\em Proceedings of the American Mathematical Society},
  137:1363--1369, 2009.

\bibitem{ChuFrantzinakisHost}
Qing Chu, Nikos Frantzinakis, and Bernard Host.
\newblock Ergodic averages of commuting transformations with distinct degree
  polynomial iterates.
\newblock {\em Proceedings of the London Mathematical Society}, 102:801--842,
  2011.

\bibitem{RueLazaro}
Thierry de~la Rue and Jose de~Sam~Lazaro.
\newblock Une transformation g\'{e}n\'{e}rique peut \^{e}tre ins\'{e}r\'{e}e
  dans un flot.
\newblock {\em Annales de l'Institut Henri Poincare (B) Probability and
  Statistics}, 39:121--134, 2003.

\bibitem{EinsiedlerWard}
Manfred Einsiedler and Thomas Ward.
\newblock {\em Ergodic Theory with a view towards Number Theory}.
\newblock Springer, 2011.

\bibitem{EisnerKrause}
Tanja Eisner and Ben Krause.
\newblock ({U}niform) convergence of twisted ergodic averages.
\newblock {\em Ergodic Theory and Dynamical Systems}, 36:2172--2202, 2016.

\bibitem{EisnerZorin}
Tanja Eisner and Pavel Zorin-Kranich.
\newblock Uniformity in the {W}iener-{W}intner theorem for nilsequences.
\newblock {\em Discrete and Continuous Dynamical Systems}, 33:3497--3516, 2013.

\bibitem{FrantzinakisZorin}
Nikos Frantzinakis and Pavel Zorin-Kranich.
\newblock Multiple recurrence for non-commuting transformations along
  rationally independent polynomials.
\newblock {\em Ergodic Theory and Dynamical Systems}, 35:403--411, 2015.

\bibitem{FurstenbergOriginal}
H~Furstenberg.
\newblock Ergodic behavior of diagonal measures and a theorem of {S}zemerédi
  on arithmetic progressions.
\newblock {\em Journal d'Analyse Math\'{e}matique}, 31:204--256, December 1977.

\bibitem{Furstenberg}
H.~Furstenberg, Y.~Katznelson, and D.~Ornstein.
\newblock The {E}rgodic {T}heoretical {P}roof of {S}zemeredi's {T}heorem.
\newblock {\em Bulletin of the American Mathematical Society}, 7(3):527--552,
  November 1982.

\bibitem{Glasner}
Eli Glasner.
\newblock {\em Ergodic Theory via Joinings}.
\newblock American Mathematical Society, 2003.

\bibitem{GlasnerWeiss}
Eli Glasner and Benjamin Weiss.
\newblock Relative weak mixing is generic.
\newblock 2017.
\newblock https://arxiv.org/abs/1707.06425.

\bibitem{Guiheneuf}
Pierre-Antoine Guihéneuf.
\newblock Dynamical properties of spatial discretizations of a generic
  homeomorphism.
\newblock {\em Ergodic Theory and Dynamical Systems}, 35:1474--1523, 2015.

\bibitem{HalmosPaper}
Paul Halmos.
\newblock In {G}eneral a {M}easure {P}reserving {T}ransformation is {M}ixing.
\newblock {\em Annals of Mathematics}, 45(4):786--792, October 1944.

\bibitem{HalmosLectures}
Paul Halmos.
\newblock {\em Lectures on Ergodic Theory}.
\newblock Chelsea Publishing Company, 1956.

\bibitem{HostKra}
Bernard Host and Bryna Kra.
\newblock Nonconventional ergodic averages and nilmanifolds.
\newblock {\em Annals of Mathematics}, 161:397--488, 2005.

\bibitem{HostKraMaass}
Bernard Host, Bryna Kra, and Alejandro Maass.
\newblock Complexity of nilsystems and systems lacking nilfactors.
\newblock {\em Journal d'Analyse Math\'{e}matique}, 124:261--295, October 2014.

\bibitem{KatokStepin}
A.B. Katok and A.M. Stepin.
\newblock Metric properties of homeomorphisms that preserve measure.
\newblock {\em Uspekhi Matematicheskikh Nauk}, 25:193--220, 1970.

\bibitem{King}
JLF King.
\newblock The generic transformation has roots of all orders.
\newblock {\em Colloquium Mathematicae}, 84/85:521--547, 2000.

\bibitem{Nadkarni}
M.G. Nadkarni.
\newblock {\em Spectral theory of dynamical systems}, chapter 8. Baire Category
  Theorems of Ergodic Theory, pages 51--61.
\newblock Birkhäuser, 1998.

\bibitem{Peterson}
Karl Peterson.
\newblock {\em Ergodic Theory}.
\newblock Cambridge University Press, 1989.

\bibitem{Robertson}
Donald Robertson.
\newblock Characteristic factors for commuting actions of amenable groups.
\newblock {\em Journal d'Analyse Math\'{e}matique}, 129:165--196, 2016.

\bibitem{EARobertson}
E.~Arthur Robinson.
\newblock The maximal abelian subextension determines weak mixing for group
  extensions.
\newblock {\em Proceedings of the American Mathematical Soceity}, 114:443--450,
  1992.

\bibitem{Rokhlin}
V.~Rokhlin.
\newblock A ``general'' measure-preserving transformation is not mixing.
\newblock {\em Doklady Akademii Nauk SSSR}, 60:349--351, 1948.

\bibitem{Solecki}
Slawomir Solecki.
\newblock Closed subgroups generated by generic measure automorphisms.
\newblock {\em Ergodic Theory and Dynamical Systems}, 34:1011--1017, 2014.

\bibitem{Tao}
Terence Tao.
\newblock {\em Poincare's Legacies, Part II: pages from year two of a
  mathematical blog}, chapter 2. Ergodic Theory, pages 161--355.
\newblock American Mathematical Society, 2009.

\bibitem{TaoZiegler}
Terence Tao and Tamar Ziegler.
\newblock Concatenation theorems for anti-{G}owers-uniform functions and
  {H}ost-{K}ra characteristic factors.
\newblock {\em Discrete Analysis}, 2016.

\bibitem{Zhao}
Yufei Zhao.
\newblock Szemeredi's theorem via ergodic theory.
\newblock http://yufeizhao.com/papers/szemeredi.pdf, April 2011.

\bibitem{Ziegler}
Tamar Ziegler.
\newblock Universal characteristic factors and {F}urstenberg averages.
\newblock {\em Journal of the American Mathematical Society}, 20:53--97, 2007.

\end{thebibliography}

\end{document}